\documentclass[12pt]{article}
\usepackage{amsfonts, amsmath, amssymb, amsthm, color, hyperref, enumitem, fancyhdr, relsize, tikz, graphicx, mathrsfs}
\usepackage[utf8]{inputenc}
\usepackage{lipsum} 
\usepackage{caption}
\usepackage{subcaption}

\title{\scshape 
On the commutator modulus of continuity for operator monotone functions \sffamily }
\author{\scshape \sffamily David Herrera
}

\newtheorem{prop}{Proposition}[section]
\newtheorem{corollary}[prop]{Corollary}
\newtheorem{thm}[prop]{Theorem}
\newtheorem{lemma}[prop]{Lemma}
\newtheorem{conj}[prop]{Conjecture}

\theoremstyle{definition}
\newtheorem{remark}[prop]{Remark}
\newtheorem{example}[prop]{Example}

\newcommand{\R}{\mathbb{R}}
\newcommand{\C}{\mathbb{C}}

\renewcommand{\P}{\mathbb{P}}
\newcommand{\erf}{\operatorname{erf}}

\newcommand{\CL}{\operatorname{CL}}
\newcommand{\Lip}{\operatorname{Lip}}
\newcommand{\M}{{\mathcal M}}
\newcommand{\wM}{\widehat{\mathcal M}}
\newcommand{\wL}{\widehat{L}^1}

\renewcommand{\Im}{\operatorname{Im}}
\renewcommand{\Re}{\operatorname{Re}}

\renewcommand{\H}{{\mathcal H}^{N-1}}

\newcommand{\diag}{\operatorname{diag}}

\newcommand{\bp}{\begin{pmatrix}}
\newcommand{\ep}{\end{pmatrix}}

\renewcommand{\H}{{\mathcal H}}

\newcommand{\vertiii}[1]{{\left\vert\kern-0.25ex\left\vert\kern-0.25ex\left\vert #1 
    \right\vert\kern-0.25ex\right\vert\kern-0.25ex\right\vert}}

\newcommand{\vertii}[1]{{\left\vert\kern-0.25ex\left\vert #1 
    \right\vert\kern-0.25ex\right\vert}}

\newcommand{\verti}[1]{{\left\vert #1 
    \right\vert}}

\newcommand{\dark}{}

\begin{document}
\dark
\large

\maketitle

\abstract{Let $f \geq 0$ be operator monotone on $[0, \infty)$. In this paper we prove that for any unitarily-invariant norm $|||-|||$ on $M_n(\C)$ and matrices $A, B, X \in M_n(\C)$ with $A, B \geq 0$ and $|||X||| \leq 1$,
\[|||f(A)X-Xf(B)||| \leq C f(|||AX-XB|||)\] for $C < 1.01975$. We do this by reducing this inequality to a function approximation problem and we choose approximate minimizers. This is much progress toward the conjecturally optimal value of $C=1$ which is known only in the case of the Hilbert-Schmidt norm.
When $|||-|||$ is the the operator norm $||-||$, we obtain a great reduction of the previously known estimate of $C = 1.25$.

We further prove that for $|||X||| \leq 1$,
\[|||A^{1/2}X-XB^{1/2}||| \leq 1.00891 |||AX-XB|||.\]
This is a great improvement toward the conjecture of G. Pedersen that this inequality for $|||-|||$ being the operator norm holds with $C = 1$. 

We discuss other related inequalities, including some sharp commutator inequalities.
We also prove a sharp equivalence inequality between the operator modulus of continuity and the commutator modulus of continuity for continuous functions on $\mathbb{R}$.}

\section{Introduction}

Let $B(\H)$ be the space of bounded operators on a separable Hilbert space $\H$. A unitarily-invariant norm $\vertiii{-}$ for $B(\H)$ is a norm defined on a closed ideal $\mathscr J$ in $B(\H)$ such that $\vertiii{UXV} = \vertiii{X}$ for all unitary operators $U, V\in B(\H)$ and $X \in \mathscr J$. If the unitarily-invariant norm is the operator norm denoted by $\vertii{-}$ or $\H$ is finite dimensional, then $\mathscr J = B(\H)$ and otherwise, it will be a subspace of the compact operators.

Unitarily-invariant norms satisfy the property that for all $X, Z \in B(\H)$ and $Y \in \mathscr J$,
\begin{equation}\label{unitary norm prop}
\vertiii{XYZ} \leq \vertii{X}\cdot\vertiii{Y}\cdot\vertii{Z}.
\end{equation}
In this paper, if a unitarily-invariant norm that we are considering is not the operator norm, then the operators considered will always act on a separable Hilbert space $\H$. Unless stated otherwise, the reader should assume that $\H$ is finite dimensional (and hence $B(\H) = M_n(\C)$) if a unitarily-invariant norm that we are considering is not the operator norm. We do this primarily to simply the exposition.

The operator norm and the Schatten-$p$ norms are unitarily-invariant norms. For $X \in M_n(\C)$, let $\sigma_j(X)$ be the decreasingly-indexed singular values of $X$. The Ky Fan norms of $X$ are defined as $\vertii{X}_{(k)} = \sum_{j=1}^k \sigma_j(X)$. The Ky Fan norms are unitarily invariant and Ky Fan's theorem states that $\vertii{X}_{(k)} \leq \vertii{Y}_{(k)}$ for each $k$ if and only if $\vertiii{X} \leq \vertiii{Y}$ for all unitarily-invariant norms. In particular, $\sigma_j(X) \leq \sigma_j(Y)$ for all $1\leq j \leq n$ is sufficient for $\vertiii{X} \leq \vertiii{Y}$ for all unitarily-invariant norms but is not necessary. 

Note the following generalization of Ky Fan's theorem: Consider the norms $\vertii{Y}_\alpha = \sum_{j=1}^n \alpha_j \sigma_j(Y)$ where $\alpha_{j} \geq \alpha_{j+1}$ and $\alpha_j \geq 0$. These are non-negative linear combinations of Ky Fan norms.  Corollary 3.5.1 of \cite{Horn Johnson Topics} states that if $g(x) \geq 0$ for $x \in [0, \infty)^m$ is increasing in each argument and $X, X_1, \dots, X_m \in M_n(\C)$ then
\[\vertiii{X} \leq g(\vertiii{X_1}, \dots, \vertiii{X_m})\]
holds for all unitarily invariant norms $\vertiii{-}$ if and only if it holds for all $\vertii{-}_\alpha$ norms.

\vspace{0.1in}

We now review the terminology and results discussed in \cite{FPUSAO}. See also \cite{OMC} and \cite{OCMC}.
Let $f$ be a continuous function on $\R$ with modulus of continuity $\omega_f(\delta) = \sup \{|f(x)-f(y)|: |x-y| \leq \delta\}$. Here are two operator generalizations:
\[\Omega_f(\delta) = \sup \{\vertii{f(A)-f(B)}: A^\ast = A, B^\ast = B, \vertii{A-B} \leq \delta\}\]
\[\Omega_f^\flat(\delta) = \sup \{\vertii{f(A)X-Xf(B)}: A^\ast = A, B^\ast = B, \vertii{X} \leq 1, \vertii{AX-XB} \leq \delta\},\]
where $A, B, X \in B(\H)$ for any $\H$.
The functions of $\delta$: $\Omega_f(\delta)$ and $\Omega_f^\flat(\delta)$ are monotonically increasing functions from $(0, \infty)$ into $[0, \infty]$. They satisfy the property that $\delta^{-1}\Omega_f(\delta)$ and $\delta^{-1}\Omega_f^\flat(\delta)$ are non-decreasing so one can show that $\Omega_f(\delta)$ and $\Omega_f^\flat(\delta)$ are continuous and subadditive.

The value of $\Omega_f^\flat$ is unchanged with restricting to one or both of: $B = A$ or $X$ self-adjoint. Note that $\omega_f \leq \Omega_f$ because we can choose $A, B \in \R$ and $\Omega_f \leq \Omega_f^\flat$ because we can choose $X = I$. 
If one reduces the definition of $\Omega_f^\flat$ by requiring that $A = B$ and $X=U$ be unitary then its value changes to $\Omega_f$ because
\begin{align}
\vertii{f(A)U-Uf(A)} &= \vertii{U(U^\ast f(A)U-f(A))} = \vertii{ f(U^\ast AU)-f(A)} \nonumber \\
&\leq \Omega_f(\vertii{U^\ast AU-A}) = \Omega_f(\vertii{AU-UA}). \label{unitary conjugation}
\end{align}

This fact provides a way of obtaining an upper bound for $\Omega^\flat_f$ in terms of $\Omega_f$ as follows. Based on the calculations in Theorem 5.3 of \cite{OCMC}, one can use a block self-adjoint operator $\mathcal A = A \oplus A$ and for any $\tau \in (0,1)$ a block unitary matrix $\mathcal U$ with upper left block being $\tau X$ to obtain the following: for any $A, X$ self-adjoint and $\vertii{X} \leq 1$, 
\begin{equation}\label{block unitary}
\vertii{[f(A)X-Xf(A)} \leq \tau^{-1}\Omega_f\left((\tau+\tau^2(1-\tau^2)^{-1/2})\vertii{AX-XA}\right).
\end{equation}
Then \cite{OCMC} uses $\tau = 1/2$ and the fact that $\Omega_f$ is monotonically increasing  to obtain $\Omega^\flat_f \leq 2\Omega_f$. However, if one instead uses $\tau = 0.582$, we obtain the slightly sharper bound $\Omega^\flat_f \leq 1.72\Omega_f$. 

Note that for a closed subset $\mathfrak F$ of $\R$, we can define $\Omega_{f, \mathfrak F}, \Omega^\flat_{f, \mathfrak F}$ to require the spectrum of $A$ and $B$ belong to $\mathfrak F$. All these results extend.

\vspace{0.1in}

This currrent paper addresses two questions. The first question is what is the optimal constant for the inequality: $\Omega^\flat_f(\delta) \leq Const.\Omega_f(\delta)$. We show that this constant is $\csc(1)\approx 1.1884$. The second and more involved question is the open problem of whether $\Omega_f^\flat(\delta) = f(\delta)$ when $f \geq 0$ is operator monontone on $[0,\infty)$ with $f(0)=0$ and also a generalization of this to unitarily-invariant norms. 

We now discuss $\Omega_f$ for $f \geq 0$ operator monotone on $[0, \infty)$.
Kittaneh and Kosaki proved the following:
\begin{thm} \label{Kittaneh Kosaki} (\cite{Kittaneh Kosaki})
Suppose that $A, B \geq 0$ are matrices in $M_n(\C)$. Let $\vertii{-}$ be the operator norm on $M_n(\C)$ and $f$ be a non-negative matrix monotone function on $[0, \infty)$ with $f(0) = 0$.
Then 
\[\vertii{f(A) - f(B)} \leq f(\vertii{A-B}).\]
\end{thm}
This shows that $f = \omega_f = \Omega_f$ for $f\geq 0$ operator monotone. Note that if one allows $f(0) > 0$ then these identities hold for $f(x)-f(0) < f(x)$.

As remarked in \cite{Kittaneh Kosaki}, this result cannot be extended in the form $\vertiii{f(A) - f(B)} \leq f(\vertiii{A-B})$ for more general unitarily-invariant norms even if $B = 0$. As an example, if we consider $f(x) = x^{r}$ and $A \geq 0$ then 
\[\vertii{A^{r}}_{(k)}=\sum_{j=1}^k \sigma_j(A)^{r} \leq k^{1-r}\left(\sum_{j=1}^k \sigma_j(A)\right)^{r} = \vertii{I}_{(k)}^{1-r}\vertii{A}_{(k)}^{r},\]
with equality when $A$ is a multiple of the identity (Exercise 1.14(d) of \cite{Grafakos}).

Noticing that $f(\vertii{A-B}) = \vertiii{f(|A-B|)}$, we see that the following inequality by Ando generalizes Kittaneh and Kosaki's above inequality.
\begin{thm}\label{Ando}
(\cite{Ando Norms of differences of f})
Suppose that $A, B \geq 0$ are matrices in $M_n(\C)$. Let $\vertiii{-}$ be a unitarily invariant norm on $M_n(\C)$ and $f$ be a non-negative matrix monotone function on $[0, \infty)$.
Then 
\[\vertiii{f(A) - f(B)} \leq \vertiii{f(|A-B|)}.\]
\end{thm}
We then have as a corollary
\begin{corollary}\label{Ando Norms of differences of f}
Suppose that $A, B \geq 0$ are matrices in $M_n(\C)$. Let $\vertiii{-}$ be a unitarily invariant norm on $M_n(\C)$ and $f$ be a non-negative matrix monotone function on $[0, \infty)$.
Then 
\[\vertiii{f(A) - f(B)} \leq \vertiii{I}f\left(\frac{\vertiii{A-B}}{\vertiii{I}}\right).\]
\end{corollary}

by using the following lemma
\begin{lemma}\label{Jensen}
Suppose that $Y \in M_n(\C)$ and $f \geq 0$ is a concave function on $[0, \infty)$. Then for any unitarily-invariant norm $\vertiii{-}$ on $M_n(\C)$, 
\[\vertiii{f(|Y|)} \leq \vertiii{I}f\left(\frac{\vertiii{Y}}{\vertiii{I}}\right).\]
\end{lemma}
\begin{proof}
Because $g(x,y) = xf(y/x) \geq 0$ on $[0,\infty)^2$ is increasing in each argument, Corollary 3.5.1 of \cite{Horn Johnson Topics} states that we need only verify
\[\vertii{f(|Y|)}_\alpha \leq g(\vertiii{I}_\alpha,\vertiii{Y}_\alpha).\]
The desired inequality then follows from Jensen's inequality for concave functions:
\[\frac{\vertii{f(|Y|)}_\alpha}{\vertii{I}_\alpha} = \frac{1}{\sum_j \alpha_j}\sum_j \alpha_j f(\sigma_j(Y)) \leq f\left(\frac{1}{\sum_j \alpha_j}\sum_j \alpha_j\sigma_j(Y)\right) =  f\left(\frac{\vertii{Y}_\alpha}{\vertii{I}_\alpha}\right).\]
\end{proof}
Note that Theorem \ref{Ando} and Corollary \ref{Ando Norms of differences of f} are identical in the case that $\vertiii{-}$ is the operator norm.
Consequently, Corollary \ref{Ando Norms of differences of f} gives a version of $\Omega_f = f-f(0)$ for any unitarily-invariant norm $\vertiii{-}$ that satisfies $\vertiii{I}= 1$.

The main subject of this paper is to what extent does Corollary \ref{Ando Norms of differences of f} extend to commutators. 
Recall the following theorem of Bhatia and Kittaneh which is a commutator version of Ando's theorem.
\begin{thm}(\cite{BK-commutators})\label{BK ineq}
Suppose that $A, B \geq 0$ and $X$ are matrices in $M_n(\C)$. Let $\vertiii{-}$ be a unitarily invariant norm on $M_n(\C)$ and $f$ be a non-negative matrix monotone function on $[0, \infty)$.
Then
\[\vertiii{f(A)X-Xf(B)} \leq \frac{1+s_1(X)^2}{2} \vertiii{f\left(\frac{2}{1+s_n(X)^2}|AX-XB|\right)}.\]
\end{thm}

As a corollary, by scaling $X$, one obtains
\begin{thm} (\cite{BK-commutators})\label{BK ineq 5/4}
Suppose that $A, B \geq 0$ and $X$ are matrices in $M_n(\C)$ and $\|X\| \leq 1$. Let $\vertiii{-}$ be a unitarily invariant norm on $M_n(\C)$ and $f$ be a non-negative matrix monotone function on $[0, \infty)$.
Then
\begin{equation}\label{BK Ineq}
\vertiii{f(A)X-Xf(B)} \leq \frac54\vertiii{f(|AX-XB|)}.
\end{equation}
\end{thm}
\begin{remark}\label{ineq relationship}
Note that by Lemma \ref{Jensen}, (\ref{BK Ineq}) implies
\[\vertiii{f(A)X-Xf(B)} \leq \frac54\vertiii{I}f\left(\frac{\vertiii{AX-XB}}{\vertiii{I}}\right).\]
Because $x\mapsto xf(c/x)$ is monotonically increasing and
\[\vertiii{X} = \vertiii{XI} \leq \vertii{X}\vertiii{I} \leq \vertiii{I},\] 
we see that the inequality above is weaker than
\begin{equation}\label{using Jensen}
\vertiii{f(A)X-Xf(B)} \leq \frac54\vertiii{X}f\left(\frac{\vertiii{AX-XB}}{\vertiii{X}}\right),
\end{equation}
where requirement that $\vertii{X} \leq 1$ can be removed because (\ref{using Jensen}) is scaling invariant. 
\end{remark}

The proof of Theorem \ref{BK ineq} makes use of several reductions. The first is the standard reduction using block matrices:
\begin{equation}\label{doubling}
\tilde A = \bp A & 0 \\ 0 & B\ep \mbox{ and } \tilde X = \bp 0 & X \\ X^\ast & 0 \ep
\end{equation}
to the case that $A = B$ and $X$ is self-adjoint. So, the inequalities are involving the norm of $AX-XA = [A,X]$ and the norm of $f(|[A,X]|)$.
The second reduction is to the case that $X$ is unitary using the Cayley transform. The case that $X$ is unitary is straight-forward and has constant $C = 1$ instead of $\frac54$ as discussed in Section \ref{Reduction to X unitary} whose calculation is essentially that of (\ref{unitary conjugation}). 

The framework of this proof of Theorem \ref{BK ineq} can already be found in Pedersen's short paper (\cite{Pedersen Commutator Ineq}) on the inequality
\begin{align}\label{power ineq}
\|[A^r,X]\| \leq \gamma_0(r)\|X\|^{1-r}\|[A,X]\|^r
\end{align}
for $0 < r < 1$. We let $\gamma_0(r) \geq 1$ be the optimal constant for which (\ref{power ineq}) holds. 

The relevant conjecture for generalizing Ando's theorem to commutators is:
\begin{conj}\label{Conjecture 0}
Suppose that $A, B, X \in B(\H)$. We require that $A, B \geq 0$ be compact with finite $\vertiii{-}$ norm and $\|X\| \leq 1$. 
Let $f$ be a non-negative matrix monotone function on $[0, \infty)$. Then
\begin{equation}\label{C0Ineq}
\vertiii{f(A)X-Xf(B)} \leq C\vertiii{f(|AX-XB|)}
\end{equation} holds for $C=1$.
If $\vertiii{-}$ is the operator norm, then $A, B$ are not required to be compact.
\end{conj}
Because $f(x) = x^r$ for $r \in (0,1)$ is operator monotone, whatever the optimal value of $C$ is, we will have $\gamma_0(r) \leq C$. Because \[\delta^r = \omega_{x^r}(\delta)= \Omega_{x^r}(\delta) \leq \Omega^\flat_{x^r}(\delta) \leq \gamma_0(r)\delta^r,\]
the conjecture of $\gamma_0(r) = 1$ implies that $\Omega_{x^r}^\flat$ is the function $x^r$. In fact, by scaling $A$, we see that $\Omega^\flat_{x^r}(\delta) = \gamma_0(r)\delta^r$.

\vspace{0.1in}

We now discuss the obtained values for the constant $\gamma_0$ in the literature. 
Olsen and Pedersen in
\cite{Pedersen Corona Construction} give a Taylor series argument to show that $\gamma_0(r) \leq (1-r)^{r-1}$. This gives $\gamma_0(1/2) \leq \sqrt{2}\approx 1.41$. They state that Davidson had previously obtained the same result with a different method.
They also make reference to numerical simulations suggesting the $\gamma_0(1/2) = 1$ is optimal. Pedersen in \cite{Pedersen Commutator Ineq} states that he ``firmly believes'' that this holds ``as in'' the case where $X$ is unitary. 
However, by our discussion of the relationship between $\Omega_f$ and $\Omega^\flat_f$ in Section \ref{Reduction to X unitary}, we should exercise some caution about these values above $\csc(1) \approx 1.1884$. 

Boyadzhiev in \cite{Boyadzhiev} showed that $\gamma_0(r) \leq \frac{\sin(\pi r)}{\pi r(1-r)}$, which gives $\gamma_0(1/2) \leq 4/\pi\approx 1.2732$.
Using the method we mentioned above in the proof of Theorem \ref{BK ineq 5/4}, Pedersen, in the article \cite{Pedersen Commutator Ineq} we mentioned above, showed that $\gamma_0(r) \leq 5/4=1.25$ and by optimizing the scaling of $X$ for each $r$, he also obtained
\begin{equation}\label{P gamma exp}
\gamma_0(r) \leq 2^r(1-r)^{-\frac{1}{2}(1-r)}(1+r)^{-\frac12(1+r)}.
\end{equation}
This estimate gives $\gamma_0(1/2) \leq 2^{\frac{3}{2}}3^{-\frac{3}{4}} < 1.2409$. This is the current smallest upper bound for $\gamma_0(1/2)$ without any further restriction on $A$ or $X$. Later Jocić, Lazarević, and Milo\v{s}ević (Theorem 3.4, \cite{Jocic Ineq derivations monotone}) extended (\ref{BK Ineq}) for $f(x) = x^r$, $r \in (0,1)$ with the constant in (\ref{P gamma exp}) to more general unitarily-invariant norms.

In \cite{Pedersen Corona Construction}, Pedersen provided a more complicated Taylor series argument to show that $\gamma_0(r) \leq m^{r}\prod_{k=2}^{m-1}\left(1-\frac{r}{k}\right)$ when $\|[A,X]\| \leq \frac 1m$ where $\vertii{X}=1, A \geq 0, \vertii{A}=1$. Pedersen gives a lower bound for this estimate as $m \to \infty$ as $((1-r)\Gamma(1-r))^{-1}$, where $\Gamma(s)$ is the Gamma function. 
The limit of the bounds for $r = 1/2$ is provided to be $\frac{2}{\sqrt{\pi}} < 1.1285$. 

As can be seen by graphing the finite product, for $\|[A,X]\| \leq 0.01$, this estimate provides $\gamma_0(r) < 1.135$ for all $r \in (0,1)$ but the estimate does not improve much more as $\|[A,X]\|\to 0$. 
As discussed below, Loring and Vides in \cite{Loring Vides} proved $\gamma_0(1/2)=1$ holds when $\|A\| = 1$ and $\|[A,X]\| \geq 1/4$.

\vspace{0.1in}

The following related conjecture will be the main focus of this paper:  
\begin{conj}\label{Conjecture 1}
Suppose that $A, B, X \in B(\H)$. We require that $A, B\geq 0$ and $X$ be compact with finite $\vertiii{-}$ norm. 
Let $f$ be a non-negative matrix monotone function on $[0, \infty)$. Then
\begin{equation}\label{CIneqMod0}
\vertiii{f(A)X-Xf(B)} \leq C\vertiii{X}f\left(\frac{\vertiii{AX-XB}}{\vertiii{X}}\right)
\end{equation} holds for $C=1$. If $\vertiii{-}$ is the operator norm, then $X$ is not required to be compact.
\end{conj} 
Note that the inequality (\ref{CIneqMod0}) can be restated for $\vertiii{X} \leq 1$ as
\begin{equation}\label{CIneqMod}
\vertiii{f(A)X-Xf(B)} \leq Cf\left(\vertiii{AX-XB}\right).
\end{equation}
This is because all non-negative operator monotone functions $f$ are concave so it is the case that if $0<x \leq 1$ and $c > 0$ then $xf(c/x) \leq f(c)$. Consequently, Conjecture \ref{Conjecture 0} and Conjecture \ref{Conjecture 1} are identical when only considering $\vertiii{-}$ being the operator norm.

The key distinctions between these two conjectures are whether $X$ is taken in the $\vertii{-}$ or $\vertiii{-}$ norms and whether $f$ is inside or outside the norm on the right-hand side of the equation.
Based on Ando's theorem implying Corollary \ref{Ando Norms of differences of f}, nuanced by Remark \ref{ineq relationship}, it is reasonable to expect Conjecture \ref{Conjecture 0} to be harder (and perhaps even less true) than Conjecture \ref{Conjecture 1}. Conjecture \ref{Conjecture 1} will be the main subject of our study in this paper.

\vspace{0.1in}

Concerning known partial results for Conjecture \ref{Conjecture 1}, Bhatia and Kittaneh (Proposition 6, \cite{BK-commutators}) proved the inequality (\ref{CIneqMod0}) rather directly for $C = 1$ for $f(x) = x^r$ for $r \in (0, 1)$ for $\vertiii{-}$ being the Hilbert-Schmidt norm. 
Later Jocić (Theorem 3.3, \cite{Jocic}) proved (\ref{CIneqMod0}) with constant $1$ for any non-negative operator monotone function on $[0, \infty)$ for the Hilbert-Schmidt norm and $\H$ separable using double operator integrals. 

\vspace{0.1in}

In Section \ref{OMF}, we review some of the basic facts about non-negative operator monotone functions on $[0, \infty)$.

In Section \ref{Commutator Bounds}, we review some useful commutator estimates in the literature and reprove some commutator Lipschitz estimates for unitarily-invariant norms.

In Section \ref{Reduction to X unitary}, we use a different transformation of $X$ into a unitary operator to show that $\Omega^\flat_f \leq \csc(1)\Omega_f$ for an arbitrary continuous function and that this is the optimal constant in this generality. 
This provides an improvement of $C=1.25$ in (\ref{C0Ineq}) to $C = \csc(1) < 1.1884$ in the case of the operator norm.
We also obtain a smaller value of $\gamma_0(r)$ in (\ref{power ineq}). We discuss why this approach to resolving either Conjecture \ref{Conjecture 0} or Conjecture \ref{Conjecture 1} faces difficulties. Hence, for the reader looking toward the main theorem, this section may be skipped. The estimates gotten in this section are also weaker than that of the main theorem.

In Section \ref{Reductions}, we show how to reduce to the case of finding a constant $C$ so that (\ref{CIneqMod}) holds for only $f(x) = \frac{x}{x+1}$ using the integral representation of operator monotone functions. 
We further reduce (\ref{CIneqMod}) to estimating the value of $E_{f}(c)$ for the function approximation problem
\[E_{f}(c) = \inf_{g' \in \wM_+} \left(\sup_{x \geq 0}(f-g)(x) - \inf_{x \geq 0}(f-g)(x) +  cg'(0)\right),\]
where $\wM_+$ is the space of all functions that are the Fourier transform of a finite positive Borel measure on $\R$. This is an approximation of the operator Lipschitz function $f$ by a commutator Lipschitz function $g$ of a specific form whose commutator Lipschitz norm can be easily bounded.

We explain why showing that $E_{f}(c) \leq Cf(c)$ for a value of $c > 0$ implies (\ref{CIneqMod}) with this value of $C$ when $\vertiii{[A,X]}=c$ in the case that $A = B$. 
This leads to a new conjecture that implies Conjecture \ref{Conjecture 1}:
\begin{conj}\label{Conjecture 2}
Let $f$ be a non-negative function that is operator monotone on $[0, \infty)$. Then 
\[E_{f}(c) \leq Cf(c)\]
holds for all $c > 0$ when $C = 1$.
\end{conj}

In Section \ref{Examples}, we explain how \cite{Boyadzhiev} and \cite{OP Corona}'s results for $\gamma_0(r)$ can be viewed as examples of our function approximation method. We also provide a few straight-forward examples which provide other (and even new) values of the constants as illustrations of the method.

In Section \ref{Proof of Main Theorem} we prove the main theorem of this paper:
\begin{thm}\label{Main Thm}
Suppose that $A, B, X \in B(\H)$, where $\H$ is finite dimensional and $\vertiii{-}$ is a unitarily-invariant norm on $B(\H)$. We require that $A, B\geq 0$. 
Let $f$ be a non-negative matrix monotone function on $[0, \infty)$. Then
\begin{equation}\label{CIneqMod1}
\vertiii{f(A)X-Xf(B)} \leq C\vertiii{X}f\left(\frac{\vertiii{AX-XB}}{\vertiii{X}}\right)
\end{equation} holds for $C<  1.01975$. If $\vertiii{-}$ is the operator norm, then there is no restriction on $\H$. 
\end{thm}
This is a great improvement from the previously known $C = 1.25$ for the operator norm. The proof involves computing upper bounds for $E_{f}(c)$ for $f(x)= x/(x+1)$ using $g$ being a parameterized antiderivative of a Gaussian for a large number of $c$'s then stitching together a universal bound of $\vertiii{[f(A),X]}$ using some continuity estimates and corner-case estimates from Section \ref{Examples}. 

The computations of the upper bounds for $E_f(c)$ involves a non-trivial MATLAB calculation on a large list of values of $c$ because there are not closed form expressions for the optimal parameters. Approximate optimal parameters are provided in the supplemental files \texttt{ErfMin\_as.txt} and \texttt{ErfMin\_bs.txt}.

We also note that the actual constant could likely be improved by more cleverly using MATLAB optimization functions to provide better approximate optimal parameters for the Gaussians, however it is not expected that it can produce the optimal estimate of $C = 1$.

In Section \ref{Special Cases}, we discuss how focusing on the commutator modulus of continuity $\Omega^\flat_f$ of $f(x) = x/(x+1)$ can produce better estimates for other operator monotone functions using the integral representation formula. Consequently, using this analysis and the data generated from the proof of the main theorem, we obtain
\begin{thm}\label{Main Thm cor sqrt}
Suppose that $A, B, X \in B(\H)$, where $\H$ is finite dimensional and $\vertiii{-}$ is a unitarily-invariant norm on $B(\H)$. We require that $A, B\geq 0$. 
Let $f$ be a non-negative matrix monotone function on $[0, \infty)$.
\begin{equation}\label{CIneqMod1sqrt}
\vertiii{A^{1/2}X-XB^{1/2}} \leq C\vertiii{X}^{1/2}\vertiii{AX-XB}^{1/2}
\end{equation} holds for $C<  1.00891$. If $\vertiii{-}$ is the operator norm, then there is no restriction on $\H$. 
\end{thm}
This is a great improvement of the known estimates and makes the conjectural value of 
$\gamma_0(r)=1$ rather plausible. 

In Section \ref{False Positives?}, we discuss two results in the literature that may give the impression that Conjecture \ref{Conjecture 1} is true. We show that these results extend to non-negative strictly concave functions, which shows that they cannot provide positive evidence for any of the conjectures we have stated because Ando's theorem fails in that generality for the operator norm as we discuss in the section.

We also prove a very general version of Conjecture \ref{Conjecture 1} for unitarily-invariant norms where the commutator Lipschitz norm of any Lipschitz function equals their Lipschitz norm (which applies for instance to the Hilbert-Schmidt norm) that does not even require $f$ to be operator monotone. This provides an extension of (Theorem 3.3, \cite{Jocic}). This result nicely points to the fact that our function approximation method in Conjecture \ref{Conjecture 2} relies on choosing an optimal function among a set of functions whose commutator Lipschitz norm is manageable. We also prove another related inequality.

\section{Operator Monotone Functions}\label{OMF}

In this section, we briefly review some of the fundamental results of operator monotone functions on $[0, \infty)$.
Some references for the standard material is (\cite{BarrySimon - Loewner},  Chapters 1, 4, and 34) and \cite{Monotone and Convex Operator Functions} . We will use a convention of the integral representation used in \cite{BK-commutators}. 

Let $I$ be an interval. We say that a function $f:I \to \R$ is operator monotone when for any Hilbert space $\H$ and all self-adjoint $A, B \in B(\H)$ having spectrum in $I$, if $A\leq B$ then $f(A) \leq f(B)$. If this is only required to hold when $\H$ is finite dimensional with $\dim \H \leq n$ then we say that $f$ is $n$-matrix monotone. If $f$ is $n$-matrix monotone for every $n \geq 0$ then we say that $f$ is matrix monotone.

Any operator monotone function is automatically monotonically increasing, concave, and analytic. 
By a standard perturbation argument, if $f$ is operator monotone on an open or half-open interval $I$ with $f$ continuous at one of the end points $a$ of $I$, then $f$ is operator monotone on $I\cup\{a\}$. Clearly, the sum and composition of operator monotone functions are operator monotone.

A fundamental example of an operator monotone function is $f(x) = -1/x$ on $(0, \infty)$. This then provides 
\begin{equation}
f_t(x) = \frac{x}{x+t} = 1-\frac{t}{x+t}
\end{equation} 
being operator monotone on $[0,\infty)$ for any $t > 0$.

We now state a deep theorem (Loewner's Theorem) about operator monotone functions. This provides different characterizations of these functions, one as an integral representation in terms of the $f_t(x)$ and one as an analytic function with the upper-halfplane as an invariant subset. We choose the domain and codomain of $f$ from the context that we are concerned about in this paper so as to obtain the integral representation as in \cite{BK-commutators}.
\begin{thm}\label{Loewner}
Let $f: (0, \infty) \to (0, \infty)$. Then the following are equivalent
\begin{enumerate}[label = (\roman*)]
\item $f$ is operator monotone on $(0, \infty)$.

\item $f$ is matrix monotone on $(0, \infty)$.

\item There is a finite positive measure $\nu$ on $[0, \infty)$ and constants $\alpha, \beta \geq 0$ so that
\begin{equation}\label{integral representation}
f(x) = \alpha + \beta x + \int_0^\infty \frac{xt}{x+t} d\nu(t).
\end{equation}

\item $f$ extends to an analytic function on $\C \setminus (-\infty, 0]$ such that if $z \in \C$ satisfies $\Im(z) > 0$ then $\Im(f(z)) > 0$.

\end{enumerate}
\end{thm}

We also have the following relevant standard examples
\begin{thm}
Let $r \in \R$. Then $f(x) = x^r$ is operator monotone on $[0, \infty)$ if and only if $r \in [0, 1]$.  Moreover, for $r \in (0,1)$, we have the integral representation
\[x^r = \frac{\sin(r\pi)}{\pi}\int_0^\infty \frac{x}{x+t}t^{r-1}dt.\]
\end{thm}

\section{Commutator Bounds}
\label{Commutator Bounds}
In this section we lay out some essential commutator estimates that will underlie our formulation of the function approximation problem. 

Let $A, X \in B(\H)$ with $A$ self-adjoint with spectrum in the interval $[a, b]$. Let $\vertiii{-}$ be a unitarily-invariant norm on $B(\H)$. If $\vertiii{-}$ is not the operator norm, we will usually assume that $\H$ is finite dimensional.

By \cite{Kittaneh Norm Ineq sa ops}, 
\begin{equation}\label{Kittaneh ineq}
\vertiii{[A,X]} \leq (b-a)\vertiii{X}.
\end{equation}
This inequality holds if $\H$ is merely separable but we wish to avoid some of the complications of compactness and continuity that appear in later calculations in infinite dimensions so we restrict to $\H$ finite-dimensional when considering norms other than the operator norm.

Note that for the operator norm the above inequality follows from the straightforward standard calculation:
\[\vertii{[A,X]} = \vertii{[A - \frac{a+b}{2},X]} \leq 2\vertii{A - \frac{a+b}{2}}\vertii{X} \leq (b-a)\vertii{X}.\]

Now, suppose that $A \geq 0$ and $\vertiii{X} \leq 1$. Let $h(x)$ be a real-valued continuous function on $[0, \infty)$. Then by (\ref{Kittaneh ineq}),
\[ \vertiii{[h(A), X]} \leq \sup_{0 \leq x \leq \|A\|} h(x) - \inf_{0 \leq x \leq \|A\|}
 h(x) \leq \sup_{x\geq 0} h(x) - \inf_{x \geq 0} h(x).\] 
We refer to this bound as the ``trivial bound.'' We use this terminology because it uses very little information about the function $h$ other than the diameter of its range.

An entirely inequivalent bound is the  following:
\[ \vertiii{[h(X), Y]} \leq Const_{\mathfrak F}\, \vertiii{[X,Y]}\]
for all $X, Y \in B(\H)$ with $X$ being self-adjoint with spectrum in a set $\mathfrak F$.
If such a bound exists, We will call $h$ commutator $\vertiii{-}$-Lipschitz on $\mathfrak F$.
We call the smallest such constant for Const. for all $X$ and $Y$, as specified above, the commutator $\vertiii{-}$-Lipschitz norm for $h$ on $\mathfrak F$. We will denote this constant as $\vertiii{h}_{\CL(\mathfrak F)}$ and refer to it as the $\CL_{\vertiii{-}}(\mathfrak F)$ norm. 

See \cite{OMC}, \cite{OCMC}, and \cite{OLF} for more about this for the operator norm from where the analogous definitions and results for $\vertiii{-}$ stated here are sourced.

We state some general facts about the $\CL_{\vertiii{-}}$ norm that we will find useful. We follow the examples in \cite{B&R} and Section 1.1 of \cite{OLF}. By Theorem 4.8 and Theorem 4.11 of \cite{FPUSAO}, the operator $\vertii{-}$-Lipschitz norm and the commutator $\vertii{-}$-Lipschitz norm are identical and that $\vertiii{g}_{\CL} \leq \vertii{g}_{\CL}$. 

Note that it is possible that $g$ is not commutator $\vertii{-}$-Lipschitz but that it is $\vertiii{-}$-Lipschitz for some unitarily-invariant norm $\vertiii{-}$. For instance, the Hilbert-Schmidt norm has that every Lipschitz function is commutator Lipschitz with the commutator Lipschitz constant being the usual Lipschitz constant (\cite{Kittaneh Lip}). There are Lipschitz functions that are not commutator Lipschitz with respect to the operator norm.

We now discuss some commutator estimates for common functions.
Suppose that $X, Y \in B(\H)$ with $X$ self-adjoint and $\vertiii{Y} < \infty$. Then
\[\vertiii{[e^{ikX}, Y]} \leq |k|\cdot\vertiii{[X, Y]}\] we note that we only need to prove this for $k > 0$. Then using
\[\frac{d}{dt}e^{itX}Ye^{-itX} = e^{itX}[X,Y]e^{-itX}\] 
we have
\begin{align*}
\vertiii{[e^{ikX}, Y]} &= \vertiii{(e^{ikX}Ye^{-ikX}-Y)e^{ikX}} = \vertiii{\int_0^ke^{itX}[X,Y]e^{-itX}dt} \\
&\leq \int_0^k\vertiii{e^{itX}[X,Y]e^{-itX}}dt = k\vertiii{[X,Y]}.
\end{align*}

The next commutator inequality is for the inverse function. If $X$ is self-adjoint with $X \geq a > 0$ then
\[\vertiii{[X^{-1},Y]} = \vertiii{X^{-1}[X,Y]X^{-1}} \leq \|X^{-1}\|^2\vertiii{[X,Y]} \leq a^{-2}\vertiii{[X,Y]}.\]

By Example 2 in Section 1.1 of \cite{OLF}, if $\log$ is the principal branch of the logarithm, then $\log(1+ix) = \int_0^\infty \frac{1}{t+1} - \frac{1}{t+ix+1}dt$. However, the integral $\int_0^\infty \frac{1}{t+1} - \frac{1}{t+z}dt$ converges absolutely and uniformly on compact subsets of $\C \setminus (-\infty, 0]$. 
In particular, for such a compact set $K$ there is a constant $d > 0$ so that for all $z \in K$, $|t+z| \geq d$ for all $t \geq 0$. So, for $t > |\Re(z)\,|+1$,
\[\verti{\frac{1}{t+1} - \frac{1}{t+z}} = \frac{|z-1|}{|(t+1)(t+z)|} \leq \frac{|z-1|}{(t+1)(t-|\Re(z)|)} \leq \frac{|z-1|}{(t-|\Re(z)|)^2}\] and for $0 \leq t \leq  |\Re(z)\,|+1$,
\[\verti{\frac{1}{t+1} - \frac{1}{t+z}} \leq 1 + \frac{1}{d}.\]
So, we see that by analytic continuation,
\[ix = \log(e^{ix}) = \int_0^\infty \frac{1}{t+1} - \frac{1}{t+e^{ix}}dt\] 
for $|x| < \pi$ and it converges absolutely for compact subintervals of $(-\pi, \pi)$.

So, suppose that $X, Y \in B(\H)$ with $X$ self-adjoint, $\|X\| < \pi$, and $\vertiii{Y} < \infty$ then
\begin{align*}
i[X,Y] &= [\log(e^{iX}),Y] = -\int_0^\infty [(t+e^{iX})^{-1},Y]dt
\end{align*}
Taking norms, we obtain
\begin{align}
\vertiii{[X,Y]} &\leq \left(\int_0^\infty \vertii{(t+e^{iX})^{-1}}^2 dt\right) \vertiii{[e^{iX},Y]} \nonumber \\
&= \left(\int_0^\infty \frac1{(t+\cos(\|X\|))^2+\sin(\|X\|)^2} dt\right) \vertiii{[e^{iX},Y]} \nonumber \\
&= \left(\int_{\cos(\|X\|)}^\infty \frac1{t^2+\sin^2(\|X\|)} dt\right) \vertiii{[e^{iX},Y]} \nonumber \\
&= \frac{\|X\|}{\sin(\|X\|)}\vertiii{[ e^{iX},Y]}. \label{log Lip}
\end{align}
This provides a sort of equivalence of commutators with $X$ and with $e^{iX}$:
\begin{equation}\label{exp equivalence}
\vertiii{[e^{iX}, Y]} \leq \vertiii{[X, Y]} \leq \frac{\|X\|}{\sin(\|X\|)}\vertiii{[e^{iX},Y]}.
\end{equation}
This gives commutator Lipschitz norms of these functions, but the operator and commutator moduli of continuity are known:
\begin{thm}(Theorem 5.9, 5.10, \cite{FPUSAO}) Let $f(x) = c_1e^{ikt}+c_2e^{-ikt}$ for $c_1, c_2 \in \C$. Then
\[\omega_f(\delta) = \Omega_{f}(\delta) = (|c_1|+|c_2|)\beta_k(\delta)\]
and
\[\Omega^\flat_{f}(\delta) = (|c_1|+|c_2|)\min\left(k\delta, 2\right),\]
where
\[\beta_k(\delta) = \left\{ \begin{array}{ll}
2\sin\left(\frac{k}2\delta\right), & 0 < \delta \leq \frac\pi{k}\\
2, & \delta > \pi/k
\end{array}\right..\]
\end{thm}
This shows in particular that 
for $f(x)$ equalling $e^{ix}$, $\cos(x)$, or $\sin(x)$:
\[\Omega_{f}(\delta) = \beta_1(\delta), \;\; \Omega^\flat_{f}(\delta) = \min\left(\delta, 2\right)\]
and $\Omega_{f}(\delta) < \Omega^\flat_{f}(\delta)$ for all $\delta \in (0,\pi)$. Moreover, the best constant $\Omega^\flat_{f}(\delta) \leq Const.\Omega_{f}(\delta)$ for these three functions $f$ is 
\begin{equation}\label{best operator and commutator ratio}
\sup_{\delta > 0} \frac{\Omega^\flat_{f}(\delta)}{\Omega_{f}(\delta)}=\frac{\Omega^\flat_{f}(2)}{\Omega_{f}(2)} = \csc(1).
\end{equation} 
This provides the lower bound in the tight estimate in Theorem \ref{sin(t) estimate}.

\vspace{0.1in}

We now merge together some of the results discussed in Section 3 of \cite{OMC} and Section 1.1 of \cite{OLF} with slightly different notation.
We use the convention of the Fourier transform: $\hat{f}(k) = \int_{\R} f(x)e^{-ixk}dx$. Let $\M$ denote the space of all signed Borel measures on $\R$ and $\M_+$ the subspace of all finite positive Borel measures on $\R$.
Let $\wL$ denote the space of functions that are the Fourier transform of a function in $L^1(\R)$.
Let $\wM$ denote the space of functions that are the Fourier transform of a measure in $\M$. If $h \in \wM$ there is a unique measure $\mu \in \M$ so that $\hat{\mu} = h$. We then define the norm 
\[\|h\|_{\wM} = \|\mu\|_{\M}.\]
We define $\|h\|_{\wL}$ similarly.

We have the following two results in \cite{OLF} for operator $\vertii{-}$-Lipschitz functions. For completeness, we present enough details from their proofs in \cite{OLF} to show that they immediately extend to $\vertiii{-}$-Lipschitz functions. Note that similar estimates are known (\cite{Boyadzhiev norm estimates for commutators}, Section 4 of \cite{Bhatia and Sinha}, Theorem 2.1 of  
\cite{Aujla completely monotone}, (2) of \cite{Brualdi editor in chief}, Theorem 4.8 and Theorem 4.11 of \cite{FPUSAO}) and what we state is a corollary of those results. 
\begin{lemma}\label{FT lemma}
Let $f \in C^1(\R)$ be such that $f' \in \wM$. Then $\vertiii{f}_{\CL(\R)} \leq \|f'\|_{\wM}$.

If $f' \in \wM_+$ then $\|f'\|_{\wM} = f'(0)$. 
\end{lemma}
\begin{proof}
By the calculation in \cite{OLF}, if $\mu \in \M$ such that $f' = \hat{\mu}$ then
\[f(x) -f(0) = i\int_\R \frac{e^{-itx}-1}{t}d\mu(t).\]
So,
\[\vertiii{[f(X), Y]} \leq \int_\R \vertiii{\frac1t[e^{-itX}-1, Y]}d|\mu|(t)\leq \int_\R \vertiii{[X, Y]}d|\mu|(t) = \|f'\|_{\wM}\vertiii{[X, Y]}.\]
When $\mu \in \M_+$, then
\[\int_\R d|\mu|(t) = f'(0).\]
\end{proof}

Now, \cite{OMC} attributes the following theorem to P{\'o}lya:
\begin{lemma}
Let $f \geq 0$ be a continuous even real-valued function on $\R$ such that $\lim_{x \to \infty} f(x) = 0$ and $f$ is convex on $[0,\infty)$. Then $f \in \wL$ and $\|f\|_{\wL} = f(0)$.
\end{lemma}
So we have the following lemma that we will make use of.
\begin{lemma}\label{convex derivative lemma}
Let $g \in C^1(\R)$ be an odd function such that $g'(x) \geq 0$ is convex on $[0, \infty)$ and that $g'(x)$ is decreasing to zero as $x \to \infty$. Then $\vertiii{g}_{\CL(\R)} \leq g'(0)$ for any unitarily-invariant norm $\vertiii{-}$ on $M_n(\C)$. 
\end{lemma}

\section{Reduction to $X$ unitary}
\label{Reduction to X unitary}
In this section we demonstrate how the method of reducing to the case of $X$ unitary can be modified to provide a better constant in the case of the operator norm. In a sense, it appears that this section provides a dead-end as far as attempting to reduce the value of $C$ lower. For the reader interested in the Main Theorem, this section may safely be passed over.

We first outline how one might try to extend Bhatia and Kittaneh's method to perhaps its best estimate. Recall that
\begin{equation}\label{Cfactors}
\vertiii{[f(A),X]} \leq \frac{1+s_1(X)^2}{2} \vertiii{f\left(\frac{2}{1+s_n(X)^2}|[A,X]|\right)}
\end{equation}
is proved first by showing that 
\begin{equation}\label{U ineq}
\vertiii{[f(A),U]} \leq  \vertiii{f\left(|[A,U]|\right)}
\end{equation}
for $U$ unitary. Note that (\ref{U ineq}) follows from Ando's theorem because
\begin{align*}
\vertiii{f(A)U - Uf(A)} 
&= \vertiii{(f(A) - Uf(A)U^\ast)U} 
= \vertiii{f(A) - f(UAU^\ast)}\\ 
&\leq  \vertiii{f\left(|A - UAU^\ast|\right)} =  \vertiii{f\left(|AU - UA|\right)}.
\end{align*}

For $X$ self-adjoint, we apply the above inequality for $U$ being the Cayley transform of $X$. This then provides conversion factors between commutators of $X$ and commutators of $U$ with another matrix which appear inside and outside the function $f$ in (\ref{Cfactors}). Note that the operator monotonicity of $f$ and an analysis of the singular values of the commutator with the Cayley transform are used in obtaining the term
\[\vertiii{f\left(\frac{2}{1+s_n(X)^2}|[A,X]|\right)}.\]

We now review how (\ref{Cfactors}) is used to obtain
\[\vertiii{f(A)X-Xf(B)} \leq \frac54\vertiii{f(|AX-XB|)}.\]
One notes that $s_n(X)\geq 0$ and that it might equal zero. So, in order to remove the constant factor inside $f$, we apply this lemma instead to $tX$ for $t>0$ to obtain
\[t\vertiii{[f(A),X]} \leq \frac{1+t^2s_1(X)^2}{2} \vertiii{f\left(\frac{2t}{1+t^2s_n(X)^2}|[A,X]|\right)}\]
 so
\begin{align}\label{scaled}
\vertiii{[f(A),X]} \leq \frac{1+t^2\|X\|^2}{2t} \vertiii{f(\,2t|[A,X]|\,)}.
\end{align}
Taking $t = 1/2$ removes the factor inside $f$ and makes the factor outside equal $5/4$ when $\|X\| \leq 1$.

In the special case that $f(x) = x^r$, one can choose $t$ differently to optimize the overall constant factor because $f$ is $r$-homogeneous. With a general operator monotone function $f$, we are helpless to remove the inner constant so we must essentially just discard it.

Consider the following result.
\begin{thm}\label{sin(t) estimate}
Suppose that $f$ is a continuous function a closed subset $\mathfrak F \subset \R$. Suppose that $A, B, X$ are operators in $B(\H)$. We require that $A$ and $B$ be self-adjoint with spectrum in $\mathfrak F$. Then for $0<\|X\|<\pi$,
\begin{equation}\label{exp ix unitary C}
\vertii{f(A)X-Xf(B)} \leq \frac{\|X\|}{\sin(\|X\|)}\Omega_{f, \mathfrak F}\left(\vertii{AX-XB}\right).
\end{equation}

Consequently, we have the following inequalities
\[\Omega_{f, \mathfrak F} \leq \Omega^\flat_{f, \mathfrak F} \leq \csc(1)\Omega_{f, \mathfrak F}.\]
The first inequality is sharp. If $\mathfrak F = \R$ then the second inequality is sharp. Note that $\csc(1)< 1.1884$.
\end{thm}
\begin{proof}
We use the standard reduction of (\ref{doubling}) discussed in the introduction to assume that $X$ is self-adjoint and $A = B$.

Let $U = e^{iX}$. By the calculation in (\ref{unitary conjugation}), since $A$ and $U^\ast A U$ have spectrum in $\mathfrak F$, we have 
\[\vertii{[f(A), e^{iX}]} \leq \Omega_{f, \mathfrak F}(\vertii{[A,e^{iX}]}).\] 
So by (\ref{exp equivalence}), we see that
\begin{align*}
\|[f(A), X]\| &\leq \frac{\|X\|}{\sin(\|X\|)} \vertii{[f(A), e^{iX}]} \leq \frac{\|X\|}{\sin(\|X\|)} \Omega_{f, \mathfrak F}(\vertii{[A, e^{iX}]}) \\
&\leq \frac{\|X\|}{\sin(\|X\|)} \Omega_{f, \mathfrak F}(\vertii{[A, X]}),
\end{align*}
because $\Omega_f$ is monotonically increasing.

When $\|X\| \leq 1$, the inequality $\Omega^\flat_{f, \mathfrak F} \leq \csc(1)\Omega_{f, \mathfrak F}$ follows because $\frac{x}{\sin(x)}$ is strictly increasing on $[0, 1] \subset [0, \pi).$
The fact that the first inequality is strict follows from the trivial example of $f(x)=x$ and the fact that the second inequality is strict follows from (\ref{best operator and commutator ratio}).
\end{proof}
\begin{remark}
The method of transforming the self-adjoint $X$ into a unitary to prove an inequality of the form (\ref{C0Ineq}), as discussed in the beginning of this section, also proves an inequality for $\Omega_f \leq Const.\Omega^\flat_f$ by the same reasoning. So, one sees that it does not seem possible to reduce the constant in (\ref{BK ineq}) below $\csc(1)$ using this similar method. 

The same type of proof used to obtain (\ref{block unitary}) by embedding $X$ self-adjoint into a block of a larger unitary matrix also will prove estimates of the form $\Omega_f \leq Const.\Omega^\flat_f$. Hence if Conjecture \ref{Conjecture 0} is true for general unitarily-invariant norms or even if (\ref{C0Ineq}) is satisfied with a constant less than $\csc(1)$ then we should search for a proof that has some nuance that would not allow $Const.$ in $\Omega_f \leq Const.\Omega^\flat_f$ to be less than $\csc(1)$ for an arbitrary continuous function.
\end{remark}
\begin{remark}
Despite having \[\vertiii{[A,e^{iX}]} \leq \vertiii{[A,X]}\] for every unitarily-invariant norm, we do not necessarily have that there is a unitary $V$ so that $V^\ast|[A,e^{iX}]|V \leq |[A,X]|$ unless the matrices $X$ and $Y$ are $2 \times 2$.
So, one cannot conclude in a straightforward way that $\vertiii{f(|[A,e^{iX}]|)} \leq \vertiii{f(|[A,X]|)}$ by the operator monotonicity. This is what breaks down in an attempt to extend the argument of our proof to a general unitarily-invariant norms to prove
\begin{align*}
\vertiii{[f(A), X]}\leq \frac{\|X\|}{\sin(\|X\|)} \vertiii{f(|[A, X]|)}
\end{align*}
which based on our current knowledge may or may not be true in general.

To see how our proof fails to prove this result a little more clearly, let $Y = \bp 2& 4& 2 \\ 4& 2& 3 \\ 2& 3& 4\ep$, $A = \bp 5& 3&  3 \\ 3& 3& 3-i \\ 3& 3+i& 5\ep$, and $X = \frac1{\|Y\|}Y$. We have that the singular values of $[A,X]$ are 
$a_j = 1.7546, 1.7036, 0.0510$ and the singular values of $[A, e^{iX}]$ are $u_j= 1.6546, 1.6027, 0.0519$. (All calculations in this remark are rounded to the fourth decimal place.) Notice that the smallest singular value of $[A, e^{iX}]$ is greater than the smallest singular value of $[A,X]$. This is what disallows us from applying the operator monotonicity of $f$ for a general unitarily-invariant norm.

In fact, if $f(x) = x/(x+0.02)$, we see that 
\[\vertii{f(|[A,e^{iX}]|)}_{(3)} = \sum_{j=1}^3 f(u_j) = 2.6976> 2.6953=\sum_{j=1}^3 f(a_j)= \vertii{f(|[A,X]|)}_{(3)}.\]
So, the desired inequality fails for the Ky Fan norm $\vertii{-}_{(3)}$, which is also just the trace norm on $M_3(\C)$.
\end{remark}

As a corollary of the previous theorem we obtain the estimates for $f(x) = x^r$:
\begin{corollary}
Suppose that $A, B, X$ are as in Theorem \ref{sin(t) estimate}. 

Then for $r \in (0,1)$,
\[ \vertii{A^rX-XB^r} \leq \tilde\gamma(r)\|X\|^{1-r}\vertii{AX-XB}^r,\]
where $\tilde\gamma(r) = \min_{t \in (0,\pi)}\frac{t^r}{\sin(t)}$.
\end{corollary}
\begin{proof} By Theorem \ref{Kittaneh Kosaki}, $\Omega_{x^r} = x^r$. We apply (\ref{exp ix unitary C}) with $X$ replaced by $tY$ where we assume that $\|Y\| = 1$  and $0 < t < \pi$.
This gives
\[t\vertii{A^rY-YB^r} \leq \frac{t}{\sin(t)}\Omega_{x^r}(t\vertii{AY-YB})= t\frac{t^r}{\sin(t)}\vertii{AY-YB}^r.\] So,
\[\vertii{A^rY-YB^r} \leq \frac{t^r}{\sin(t)}\vertii{AY-YB}^r.\]
For general $X$, the result then follows by applying this to $Y = X/\|X\|$.
\end{proof}

\begin{remark}
We briefly state here that we have proved that $\gamma_0(1/2) \leq 1.1748$ (using $t = 1.166$), which is a new estimate. We will obtain still better estimates later in the paper using our function approximation approach.
\end{remark}

\section{Reductions}\label{Reductions}
In this section we discuss reductions of the desired inequality
\begin{equation}\label{CI}
\vertiii{[f(A), X]} \leq Cf(\vertiii{[A,X]}), \;\; \vertiii{X} \leq 1.
\end{equation}
into a more manageable form.

Let $g \in \CL_{\vertiii{-}}([0, \vertii{A}])$. Then by the inequalities of Section \ref{Commutator Bounds},
\begin{align}
\vert\kern-0.65ex\vert\kern-0.25ex\vert[f(A),& X]\vert\kern-0.65ex\vert\kern-0.25ex \leq \vertiii{[(f-g)(A), X]} + \vertiii{[g(A), X]} \nonumber \\
&\leq \left(\sup_{0\leq x \leq \vertii{A}}(f-g)(x)-\inf_{0\leq x \leq \vertii{A}}(f-g)(x)\right)\vertiii{X} +  \vertiii{g}_{\CL([0, \vertii{A}])}\|[A,X]\|.\label{fund comm ineq}
\end{align}
Using $\vertiii{X} \leq 1$, we define
\begin{equation}\label{tilde E def}
\tilde{E}_{f}(c) = \inf_{g \in \CL_{\vertiii{-}}([0, \infty))} \left(\sup_{x \geq 0}(f-g)(x)-\inf_{x \geq 0}(f-g)(x) +  c\vertiii{g}_{\CL([0, \infty))}\right).
\end{equation}
Then we see that (\ref{CI}) is implied by
\begin{equation}\label{E eq}
\tilde{E}_{f}(c) \leq Cf(c),
\end{equation}
where $c = \|[A,X]\|$.

We state the result that this holds with $C = 1$ as a conjecture:
\begin{conj}\label{Conjecture 3}
Let $f$ be a non-negative function that is operator monotone on $[0, \infty)$. Then 
\[\tilde{E}_{f}(c) \leq Cf(c)\]
holds for all $c > 0$ when $C = 1$.
\end{conj}
\begin{remark}
Note that the estimate of breaking a function $f$ of $A$ into a term which gets a trivial bound and a term that gets an operator Lipschitz bound is done by Loring and Vides \cite{Loring Vides} in the case of a power series and $\|A\| \leq 1$, drawing on the power series methods of earlier papers.

Our approach  we wish to have $A \geq 0$ general and because we wish to approximate $f(x)$ by a function whose operator Lipschitz function norm is 
\end{remark}

In general $\vertiii{g}_{\CL(\R)} \leq \|g'\|_{\wM}$. We could make the assumption on $g$ that $g' \in \wM$ to obtain an upper bound for $\tilde{E}_f(c)$. We will actually restrict the choice of $g$ further to $g' \in \wM_+$ which has many concrete examples.

\vspace{0.1in}

If we make further assumptions on $g$ and $f-g$ then the expression in the infimum of the definition for $\tilde{E}_f(c)$ simplifies.
Suppose that $g$ is $C^1(\R)$ and $g'$ is the Fourier transform of a finite positive Borel measure. Then by Lemma \ref{FT lemma},
\begin{equation}\label{E g}
\tilde{E}_{f}(c) \leq E_{f}(c) := \inf_{g' \in \wM_+} \left(\sup_{x \geq 0}(f-g)(x) - \inf_{x \geq 0}(f-g)(x) +  cg'(0)\right).
\end{equation}
With these definitions, we can see that Conjecture \ref{Conjecture 2} implies Conjecture \ref{Conjecture 3} implies Conjecture \ref{Conjecture 1}.

\vspace{0.1in}

Suppose in addition to one of the above requirements, we also choose $g$ so that $f(x)-g(x) \geq 0$ on $[0, \infty)$. Then 
\begin{equation}\label{E g f-g}
E_{f}(c) = \inf_{\substack{g \in \wM_+ \\ (f - g) |_{[0,\infty)}\geq 0 } } \left(\sup_{[0, \infty)}(f-g) +  cg'(0)\right).
\end{equation}

An equivalent formulation of $E_f(c)$ is what we will call the derivative perspective. Fix a function $G$ that will be $g'$ and let $F = f'$.  We assume that $G\in \wM_+$. Without loss of generality, $f(0) = 0$ and we define $g(0)=0$ so that $f(x)=\int_0^x F(t)dt$ for $x \geq 0$ and $g(x)=\int_0^x G(t)dt$.  We can always shift $g$ afterward to make it more aesthetically or theoretically appealing, if we wish.

For simplicity, let $\P_f$ denote the set of functions $G$ with these conditions. Then,
\begin{equation}\label{E diff per}
E_{f}(c) = \inf_{G \in \P_f} 
\left( \sup_{x\geq 0}\int_0^x (F(t)-G(t)) dt +  cG(0)\right).
\end{equation}
The derivative representation in  (\ref{E diff per}) of $E_{f}(c)$, the standard representation in (\ref{E g}), or some hybrid of the two will be useful when defining approximate minimizer functions $g$ in the infimum of the functional defining $E_f(c)$. They will also be useful when calculating the obtained an upper bound for $E_f(c)$ from this $g$.

Note that $f - g \geq 0$ on $[0, \infty)$ converts into the majorization condition of $F$ and $G$: $\int_0^x (F(t)-G(t)) dt \geq 0$ for all $x \geq 0$.
\vspace{0.1in}

We now specialize to the case of operator monotone functions on $[0, \infty)$. 
If $f(x)$ is operator monotone on $[0, \infty)$ then it has the integral expression (written in the form used in \cite{BK-commutators}) as a combination of functions of the form $f_{t}(x) = \frac{x}{x+t}$:
\[f(x) = \alpha + \beta x + \int_0^\infty \frac{xt}{x+t} d\nu(t).\]
with $\alpha, \beta \geq 0$ and $\nu$ a finite positive Borel measure on $[0, \infty)$ so we will assume that $t \geq 0$ in the above integral.

Suppose that (\ref{CI}) holds for each $f_t$. Then
\begin{align}\nonumber
\vertiii{[f(A),X]}  &= \vertiii{\left[\alpha + \beta A + \int_0^\infty tf_t(A) d\nu(t), X\right]} \leq \beta\vertiii{[A,X]} + \int_0^\infty t\vertiii{[f_t(A),X]} d\nu(t) \\
&\leq C\alpha+ C\beta\vertiii{[A,X]} +C\int_0^\infty tf_t(\vertiii{[A,X]})d\nu(t) = Cf(\vertiii{A,X]}), 
\label{f_t reduction}
\end{align}
since $C \geq 1$. Then (\ref{CI}) holds for $f$.
So, it suffices to show (\ref{CI}) for $f_t(x)$ for each $t$.

We now investigate (\ref{CI}) for $f_t(x)$ a bit more. Note that
\[ \vertiii{[f_t(A), X]} \leq Cf_t(\vertiii{[A, X]})\]
means
\[ \vertiii{[A(A+t)^{-1}, X]} \leq C\frac{\vertiii{[A, X]}}{\vertiii{[A, X]}+t}\]
which is equivalent to
\begin{equation}\label{f_t scaling} 
\vertiii{[f_{at}(aA), X]} \leq Cf_{at}(\vertiii{[aA,X]})
\end{equation}
for any $a > 0$.
So, (\ref{CI}) for $f= f_t$ and all $A \geq 0$ is equivalent to (\ref{CI}) for $f=f_{at}$ for all $A \geq 0$. 
So, it suffices to show (\ref{CI}) for all values of $c = \vertiii{[A,X]}$ for $f_1(x) = \frac{x}{x+1}$.

We summarize this as two lemmas.
\begin{lemma}\label{f_t lemma}
Let $\H$ be a Hilbert space, $f_t(x) = \frac{x}{x+t}$, $C \geq 1$, and $X, A_0\in B(\H)$ with $A_0 \geq 0$.
Then the following are equivalent.
\begin{enumerate}[label = (\roman*)]
\item
$\displaystyle\vertiii{[f_1(A),X]} \leq C f_1\left(\vertiii{[A,X]}\right)$ for all $A = sA_0, s \geq 0$.
\item
$\displaystyle\vertiii{[f_t(A),X]} \leq C f_t\left(\vertiii{[A,X]}\right)$ for all $A = sA_0, s \geq 0$.
\end{enumerate}
\end{lemma}
\begin{lemma}\label{f_1 lemma}
Let $f_1(x) = \frac{x}{x+1}$ and $C \geq 1$. Let $X \in B(\H)$. 

Then if
\begin{equation}\label{f1 comm ineq}
\vertiii{[f_1(A),X]} \leq C f_1\left(\vertiii{[A,X]}\right)
\end{equation}
for all $A \in B(\H)$ with $A \geq 0$ then
\[\vertiii{[f(A),X]} \leq C f\left(\vertiii{[A,X]}\right)\] for all $A \in B(\H)$ with $A \geq 0$ and every non-negative operator monotone function $f$ on $[0, \infty)$.
\end{lemma}

We also show the following two lemmas whose purpose is to show some continuity of our bounds for $f_1$. This will allow us to use upper bounds for the smallest constant $C$ that satisfy $E_{f_1}(c) \leq Cf_1(c)$ for a finite number of values of $c$ to prove (\ref{f1 comm ineq}) for all $c$ in an interval at the expense of increasing the constant from $C$ to $D >C$ for the interpolated bounds.
\begin{lemma}\label{C continuity}
Let $f_1(x) = \frac{x}{x+1}$. Suppose that for some $c > 0$, there is a constant $C \geq 1$ such that
\begin{equation}\label{C continuity - C}
\vertiii{[f_1(A),X]} \leq C f_1\left(\vertiii{[A,X]}\right)
\end{equation}
for all $A, X \in B(\H)$ with $A \geq 0$, $\vertiii{X} \leq 1$, and $\vertiii{[A,X]} = c$. 

For $d > c$, let
\[D = C\cdot\left(\frac{d+1}{c+1}\right).\]
Then
\begin{equation}\label{C continuity - D}
\vertiii{[f_1(A),X]} \leq D f_1\left(\vertiii{[A,X]}\right)
\end{equation}
for all $A, X \in B(\H)$ with $A \geq 0$, $\vertiii{X} \leq 1$, and $\vertiii{[A,X]} = d$.
\end{lemma}
\begin{proof}
Let $A, X$ satisfy the conditions with commutator having norm $d$.
Let $Y = \frac{c}{d}X$. Then $\vertiii{[A,Y]} = c$ and $\vertiii{Y} \leq \frac{c}{d} < 1$. So, we apply the given to obtain
\begin{align*}
\vertiii{[f_1(A),X]} &= \frac{d}{c}\vertiii{[f_1(A),Y]} \leq C\cdot\left(\frac{d}{c}\right) f_1\left(\vertiii{[A,Y]}\right) \\
&= C\cdot\left(\frac{d}{c}\right)\left(\frac{c}{c+1}\right)\left(\frac{d+1}{d}\right)f_1(d) \\
&= C\cdot\left(\frac{d+1}{c+1}\right)f_1(\vertiii{[A,X]}).
\end{align*}

\end{proof}

An immediate consequence of this is the following result that we will use in the MATLAB calculations for our main theorem.
\begin{lemma}\label{C continuity points}
Let $f_1(x) = \frac{x}{x+1}$ and $0 < c_1 < \cdots < c_n$. Suppose that there are constants $C_k > 0$ such that
\begin{equation}\label{C continuity - C points}
\vertiii{[f_1(A),X]} \leq C_k f_1\left(\vertiii{[A,X]}\right)
\end{equation}
for all $A, X \in B(\H)$ with $A \geq 0$, $\vertiii{X} \leq 1$, and $\vertiii{[A,X]} = c_k$. 

Let ${\Delta}c = \max_k(c_{k+1}-c_k)$, $c_{n+1}:= c_n+{\Delta}c$, and
\[D_k = C_k\left(\frac{c_{k+1}+1}{c_k+1}\right) \leq C\left(\frac{c_1+ {\Delta}c+1}{c_1+1}\right) =: D.\]
Then for $d \in [c_k, c_{k+1}]$,
\begin{equation}\label{C continuity - D interval}
\vertiii{[f_1(A),X]} \leq D_k f_1\left(\vertiii{[A,X]}\right)
\end{equation}
for all $A, X \in B(\H)$ with $A \geq 0$, $\vertiii{X} \leq 1$, and $\vertiii{[A,X]} = d$.
\end{lemma}
\begin{remark}\label{D_k remark}
The examples that we will see in later sections will have $C_k$ close to $1$ for $c_k$ close to $0$ and $C_k$ growing until some value of $c$. See Figure \ref{piecewiseQuadraticMin_graph} and Figure \ref{ErfMin_graph1_updated}.  Also, the factor $\frac{c_{k+1}+1}{c_k}$ is largest for $k = 1$ assuming $c_{k+1}-c_k$ constant.

So, in such a case it is optimal to use the bounds $D_k$ because $\max_k D_k$ will be smaller than $D$.
\end{remark}

\vspace{0.1in}

Now let us consider the special case of $f(x) = x^r$. This will provide some useful new examples of using our method.
The nice simplifying step is to notice that the inequality
\[\vertiii{[A^r, X]} \leq \gamma(r)\vertiii{[A,X]}^r\]
is homogeneous in $A \geq 0$. So although we are not free to scale $X$ since $\vertiii{X} \leq 1$ is required, we are free to scale $A$ so as to require that $c=\vertiii{[A,X]} = 1$. So, we need only show (\ref{E eq}) for $c = 1$. 

So, we need only find a single function $G$ so that the right-hand side of 
\begin{equation}\label{E diff per x^r}
\gamma(r) \leq E_{x^r}(1) =\inf_{G \in \P_{x^r}} 
\left( \sup_{x\geq 0}\int_0^x (rt^{r-1}-G(t)) dt +  G(0)\right).
\end{equation}
can be made as small as we can make it.
Note that $F(t) = rt^{r-1}$ is unbounded as $t \to 0^+$, non-negative, convex, and decreasing to zero. It does not belong to $L^1((0, \infty))$ because $f(t) = t^r$ is not bounded as $t \to \infty$.
 
\section{Examples}\label{Examples}

In this section we provides some examples of our method. They illustrate some of the techniques that we will be using in the proof of the Main Theorem. However, only Example \ref{simple example} and Example \ref{x/(x+1) shift} are explicitly used. They are used to deal with corner cases in the same way that Example \ref{f_1 quadratic} uses Example \ref{simple example}. 

\begin{example}\label{simple example}
Note that $f(x) = f_1(x)=x/(x+t)$ has derivative $f'(x) = 1/(x+1)^2$. Extending $f(x)$ to $\R$ as an odd function then shows that $f'$ satisfies the conditions of Lemma \ref{convex derivative lemma}. So, $f$ has operator Lipschitz norm of $f_1'(0) = 1$. Because $0 \leq f \leq 1$, we then see that by choosing $g  = 0$ or $g = f$ we obtain
\[\vertiii{[f(A),X]} \leq \min(\vertiii{[A,X]}, 1).\]
Now, $\min(c,1) \leq 2f(c)$ for all $c > 0$.
So, we obtain 
\[\vertiii{[f_1(A),X]} \leq 2f_1(\vertiii{[A,X]}),\]
which provides $C = 2$ for a general operator monotone function by Lemma \ref{f_1 lemma}. This is the simplest bound we obtain. 
\end{example}

\begin{example}\label{Boyadzhiev example}
Boyadzhiev in \cite{Boyadzhiev} produced a proof of the inequality in \ref{simple example} (with $A, B, X$ belonging to a more general Banach algebra) under certain assumptions on the measure $\nu$ and an integral representation argument. 

However, despite the method in \cite{Boyadzhiev} being more complicated than necessary for this case, it does provide a stronger bound for $f(x) = x^r$.

We illustrate how the method employed actually is an example of the method using $E_f(c)$. We first recount the method done in that paper.
It involved writing the integral representation for $A^r$, breaking the integral into two pieces, then applying the trivial bound and a commutator bound appropriately on the two parts, then optimizing the breaking point.

We illustrate this  using a slightly different representation of $x^{1/2}$ that has a more familiar integral representation. Because this integral representation can be gotten by a strictly increasing change of variable from the representation of that paper, our presentation does not change the method or result.

Write 
\begin{equation}\label{sqrt int rep}
A^{1/2} = \frac1\pi \int_0^\infty A(A+t^2)^{-1}dt.
\end{equation}
Then the calculation proceeds by letting $s > 0$ and writing
\begin{align*}
[A^{1/2}, X] &= \frac2\pi \int_0^\infty [A(A+t^2)^{-1}, X]dt \\
&= \frac2\pi \int_0^s [A(A+t^2)^{-1}, X]dt + \frac2\pi \int_s^\infty [A(A+t^2)^{-1}, X]dt.
\end{align*}
We then apply the trivial bound to the first term $\|[A(A+t^2)^{-1}, X]\| \leq 1$. For the second term, we write $\frac{x}{x+t^2} = 1 - \frac{t^2}{x+t^2}$ and so
\begin{equation}\label{inv comm}
\vertiii{[A(A+t^2)^{-1}, X]} = \vertiii{t^2[(A+t^2)^{-1}, X]} \leq t^{-2}\vertiii{[A, X]} = \frac{c}{t^2}
\end{equation}
and hence
\begin{align*}
\vertiii{[A^{1/2}, X]} &\leq  \frac2\pi \int_0^s dt + \frac2\pi \int_s^\infty \frac{c}{t^2}dt = \frac{2}{\pi}\left(s + \frac{c}s\right),
\end{align*}
where $c = \vertiii{[A, X]}$.
Choosing the optimal value of $s = c^{1/2}$ provides
\[\vertii{[A^{1/2}, X]} \leq \frac4\pi \vertii{[A, X]}^{1/2}.\]
This provides the constant $\gamma(1/2) \leq \frac4\pi \approx 1.2732$.

We will now dig into this a little bit more. First, note that (\ref{inv comm}) is an immediate consequence of what we stated earlier with the operator Lipschitz bound of $f_{t^2}(x)$ because $f'_{t^2}(0) = t^{-2}$.
However, what we will now see is that the entire integral decomposition fits exactly within our function approximation paradigm.

The function $g(x) = \frac2\pi \int_s^\infty \frac{x}{x+t^2}dt$ satisfies the conditions of Lemma \ref{convex derivative lemma}. We can verify this by differentiating under the integral sign: $g'(x) = \frac2\pi \int_s^\infty \frac{t^2}{(x+t^2)^2}dt$. Seeing that higher derivatives pass through directly, we conclude.
We also note that 
\begin{align*}
g(x) &= \frac2\pi \int_s^\infty \frac{1}{1+(x^{-1/2}t)^2}dt =\frac2\pi x^{1/2}\int_{x^{-1/2}s}^\infty \frac{1}{1+u^2}du \\
&= x^{1/2}\left(1 - \frac{2}{\pi}\arctan\left(\frac{s}{x^{1/2}}\right)\right).
\end{align*}

The trivial bound for $[\frac2\pi \int_s^\infty A(A+t^2)^{-1}dt, X]$ is then just the trivial bound for $[(f-g)(A), X]$ and using the more involved bound of (\ref{inv comm}) for the integrand amounts using the Lipschitz bound for  the integral: $[g(A), X]$. This can be seen by noting that $g'(0) = \frac2\pi\int_s^\infty t^{-2}dt = \frac2\pi s^{-1}.$

So, we see that a guess for the minimization of $E_{x^{1/2}}(c)$ was obtained using that choice of $g$ with a free parameter $s$ that was optimized. The integral representation allowed us a natural way of obtaining this parameterized function $g$. However, as we will see in the latter two examples using $f(x) = x^{1/2}$, there exist considerably better choices.
\end{example}

\begin{example}\label{OP Corona example}
In \cite{OP Corona}, Olsen and Pedersen had an approach to $\vertii{[A^r, X]}$ with $\vertii{X} \leq 1$, which amounted to proving that for $d > 0$ one has the trivial bound
\[\vertii{[(A+d)^r-A^r, X]} \leq d^r\]
and also using a Taylor series argument showing that
\[\vertii{[(A+d)^r, X]} \leq rd^{r-1}\vertii{[A, X]}.\]
This choice of $g$ provides $\gamma(r) \leq (1-r)^{r-1}$ so $\gamma(1/2) \leq \sqrt{2} \approx 1.4142$.

However, this example is just that of $g(x) = (x+d)^r$ on $[0, \infty)$ extended as an odd function. Then $g$ the conditions of Lemma \ref{convex derivative lemma} with $0 \leq g(x)-f(x) \leq g(0)-f(0) = d^r$ and $g'(0) = rd^{r-1}$.

So, we see that this example is also captured in our paradigm. In a sense, $g$ is a very natural function satisfying Lemma \ref{convex derivative lemma} to choose to approximate $x^r$ because the issue is that $x^r$ is not differentiable at $0$. However, it provided a worse estimate than may be desired.

Note that all these inequalities are valid replacing the operator norm by the unitarily-invariant norm $\vertiii{-}$ throughout.
\end{example}

\begin{example}\label{x/(x+1) shift}
We can try doing the same technique for $f(x) = f_1(x)$. Let $g(x) =f(x+d)=1-\frac{1}{x+d+1}$ for $x \geq 0$, $d \geq 0$. (Following the same sort of argument in what follows for $d \in (-1, 0]$ shows that this does not produce positive results.) Extend $g$ to be odd. 

Then $-g(0) \leq f(x)-g(x) \leq 0$ for $x \geq 0$, $g(0) = 1-\frac{1}{d+1}$, and $g'(0) = \frac1{(d+1)^2}.$  So,
\begin{align*}
E_f(c) &\leq \inf_{d  \geq 0} \left(1-\frac{1}{d+1} + \frac{c}{(d+1)^2}\right) = \inf_{0<s \leq 1 } \left(1-s + cs^2\right), 
\end{align*}
where $s = (d+1)^{-1}$. The objective function $\Phi(s) = 1-s + cs^2$ is minimized at $s = \frac{1}{2c}.$ If $c \geq 1/2$ then we choose this value of $s$. If $c < 1/2$ then we choose $s=1$. Hence, with
\begin{align*}
e(c) := 
\left\{ 
\begin{array}{ll}
c, & 0 < c < 1/2 \\
1 - \frac1{4c}, & c \geq 1/2
\end{array}
\right., 
\end{align*}
we have $E_{f_1}(c) \leq e(c).$

The constant that we get from this choice of parameterized $g$ is then
\[C = \sup_{c > 0} \frac{e(c)}{f_1(c)}.\]
It can be checked that this obtained for $c = \frac23$ providing the value $C = 1.5625$. 
\end{example}

\begin{example}\label{Loring Vides example}
Loring and Vides in \cite{Loring Vides} expand on the polynomial commutator argument of the previous example and use a function approximation method similar to our own, except the standard commutator Lipschitz estimate for a non-linear polynomial will involve the norm of $A$.

We illustrate this with the following result which is (Theorem 8.1, \cite{Loring Vides}) in the case of the operator norm:
\begin{prop} \label{Loring Vides abs} Let $A, X \in M_n(\C)$ with $A$ self-adjoint and $\vertiii{X}\leq 1$. Then
\begin{equation}\label{abs ineq}
\vertiii{[|A|, X]} \leq \frac12\vertii{A}+\vertiii{[A,X]}.
\end{equation}
\end{prop}
\begin{proof}
Without loss of generality, we may scale $A$ so that $\vertii{A}=1$. We now follow the proof of (Theorem 8.1, \cite{Loring Vides}).

Let $f(x) = |x|$, $g(x) = \frac12x^2$, and $j(x) = f(x)-g(x)$. Then
\[\vertiii{[g(A),X]} = \frac12\vertiii{A[A,X]+[A,X]A} \leq \vertii{A}\vertiii{[A,X]} = \vertiii{[A,X]}.\]
Also, $\min_{x \in [-1,1]}j(x) = 0$ and $\max_{x \in [-1,1]}j(x) = \frac12$.

Then by the argument that goes into (\ref{fund comm ineq}),  
\[\vertiii{[|A|, X]} \leq \frac12\vertiii{X}+\vertiii{[A,X]}\leq \frac12+\vertiii{[A,X]}.\]
\end{proof}

\end{example}

We now proceed to provide two examples using the derivative perspective to estimating $E_{f}(c)$ for $f(x) = x^r$. We will first begin with a very general $f$ then specialize to the power function when it simplifies the exposition. 

The first example will be a simpler parameterized derivative $G=g'$ with only one parameter due to the simplifying assumptions on $G$. The second will weaken the assumptions, which introduces another parameter to provide a tighter estimate.

\begin{example}\label{F geq G example}

By (\ref{E diff per x^r}), we will produce $G \in \P_{x^r}$ by making it non-negative, even, convex on $[0, \infty)$, and $\lim_{x\to\infty}G(x) = 0$. 
For this first example we will further require that $G \leq F$ so that $\int_0^x (F(t)-G(t)) dt \geq 0$ for all $x \geq 0$ is automatic. 
We proceed in generality because we do not need to use anything about $F$ until we start to perform calculations, except that $F$ is differentiable.

For both examples would like the minimize $\sup_{x\geq 0}\int_0^x (F(t)-G(t)) dt +  cG(0)$. If we think of $G(0)$ as a fixed value $h > 0$ then because $G$ is convex and $G \leq F$, the largest that we can make $G$ is for $G$ to be a piecewise function whose graph is a line from $(0, h)$ to a point $(a, F(a))$ of the graph of $F$ so that the line is tangent to the graph of $F$. Then $G(t) = F(t)$ for $t \geq a$. This can be thought of as the portion of the Cartesian plane above the graph of $G$ being the convex hull of the graph of $F$ and the point $(0, h)$. 

Since the slope of the line segment portion of the graph of $G$ mentioned above is $\frac{F(a)-h}{a} = F'(a)$, we have $h = F(a)-aF'(a)$ and
\begin{equation}\label{G F-G pos}
G(t) = \left\{\begin{array}{ll} 
F'(a)(t-a)+ F(a),&  0 \leq t < a \\
F(t) , & t \geq a
\end{array}\right..
\end{equation}
Then
\begin{align*}
E_f(c) &\leq \inf_{a > 0}\left[\int_0^a (F(t)-F'(a)t-F(a)+aF'(a)) dt +  c\left(F(a)-aF'(a)\right)\right] \\
&= \inf_{a > 0}\left[f(a)-\frac{F'(a)}{2}a^2-aF(a)+a^2F'(a) + c\left(F(a)-aF'(a)\right)\right]\\
&= \inf_{a > 0}\left[f(a)+\frac{F'(a)}{2}a^2-aF(a) + c\left(F(a)-aF'(a)\right)\right].
\end{align*}

We now apply this to $f(t) = t^{r}, F(t)= rt^{r-1}, F'(t) = r(r-1)t^{r-2}$:
\begin{align*}
\gamma(r)\leq E_{x^r}(1) &\leq \inf_{a > 0}\left[\left(1+\frac{r(r-1)}{2}-r\right)a^r + \left(r-r(r-1)\right)a^{r-1}\right]\\
&=\inf_{a > 0}\left[\left(\frac{r^2}{2}-\frac32r+1\right)a^r + r(2-r)a^{r-1}\right]\\
&=(2-r)\inf_{a > 0}\left[\frac{1-r}{2}a^r + ra^{r-1}\right].
\end{align*}

This is minimized at $a = 2$ which gives 
\[\gamma(r) \leq (2-r)2^{r-1}.\] In particular, $\gamma(1/2) < 1.0607.$ In fact, for all $r \in (0,1)$,  $\gamma(r) < 1.062$. 
This is already a great improvement on the known results. 

Moreover, it is an improvement on the choice of $\tilde{g}(x) = f(x+\delta)$ which we can compare to this method. We see that $\tilde{g}'(0) = r\delta^{r-1}$ however the area between $f'(t)$ and $\tilde{g}'(t)$ is considerable and need not be so large if we require $\tilde{G} = \tilde{g}' \leq f'$. Our choice of $G$ in this example shows how to take advantage of the unused area in Example \ref{Boyadzhiev example} and Example \ref{OP Corona example} while maintaining the same operator Lipschitz constant of $g$, which is $G(0)$. 

\vspace{0.1in}

As far as expressing $g$ in terms of $G$, if we write $g$ using (\ref{G F-G pos}) and shift $g$ so that $g(a) = f(a)$, we see that $g$ is $f$ on $[a,\infty)$ and is the degree two Taylor polynomial of $f$ on $[0, a)$:
\begin{equation}\label{g F-G pos}
g(x) = \left\{\begin{array}{ll} 
\frac{1}{2}F'(a)(x-a)^2+ F(a)(x-a) + f(a),&  0 \leq x < a \\
f(x) , & x \geq a
\end{array}\right..
\end{equation}

\end{example}

\begin{example}\label{x^1/2 piecewise quadratic}
In this example we will make use of $\int_0^x (F(t)-G(t)) dt \geq 0$ and not assume that $G \leq F$. One way of thinking about why we would want to do this is if we think of $G(0)$ as fixed and increasing the value of $a$. This moves the ending point of the line from $(0, G(0))$ to $(a, F(a))$ along the graph of $F$ so the linear part of $G$ is no longer tangent to the graph of $F$ and begins to cross over the graph of $F$.

So, although $G$ will be crossing $F$, if $a$ is not increased too much then we will still have $f - g$ being non-negative because of the positive area already accumulated in $\int_0^x (F(t)-G(t)) dt$ for small $t$ will compensate for the new negative area introduced. Moreover, the supremum of $\int_0^x (F(t)-G(t)) dt$ will decrease. So, we see that we can reduce our bound for $E_f(c)$ without changing $g'(0)$ or the fact that $f-g \geq 0$.

Although we just described this new function $G$ in terms of $a$ and $h$, we will instead formulate $G$ in terms of $a$ and the difference $m$ between the slope of the linear portion of $G$ and the tangent line slope of $F$ at value $a$. 
$G$ will still have the same piecewise structure but instead of the line having slope $F'(a)$, it will have slope $F'(a)+m$ where $m \leq 0$. We prefer this formulation because then the requirements on $a$ and $m$ are independent: $a > 0, m \leq 0$ whereas $h \geq F(a)-aF'(a)$.

We will assume that $G(0) < F(0+)=\lim_{t \to 0^+}F(t)$, which was automatic in the prior example because $F$ is convex. With this assumption, the convexity of $F$ implies that there will be a point $t_\ast > 0$ such that $G(t) \geq F(t)$ exactly on $[t_\ast ,a].$

So, we will have
\begin{equation}\label{G F-G general}
G(t) = \left\{\begin{array}{ll} 
(F'(a)+m)(t-a)+ F(a),&  0 \leq t < a \\
F(t) , & t \geq a
\end{array}\right.
\end{equation}
and hence
\begin{equation}\label{g F-G general}
g(x) = \left\{\begin{array}{ll} 
\frac{F'(a)+m}{2}(x-a)^2+ F(a)(x-a) + f(a),&  0 \leq x < a \\
f(x) , & x \geq a
\end{array}\right..
\end{equation}
Here we have shifted $g$ so that $g(a) = f(a)$.

The tedious part is that $F-G$ can have multiple signs and we need to know where the sign changes. In particular, $F-G \geq 0$ on $[0, t_\ast]$, $F-G \leq 0$ on $[t_\ast,a]$, and $F-G = 0$ on $[a, \infty)$ and the point $t_\ast$ can be found by solving $F(t_\ast)=G(t_\ast)$. So, writing $j = f-g$, we have
\begin{align*}
E_{f}(c) \leq \sup_{[0, \infty)}j - \inf_{[0, \infty)}j + cg'(0) = j(t_\ast) - \min(j(0), j(a))+cg'(0).
\end{align*}
because $j$ is increasing on $[0, t_\ast]$ and decreasing on $[t_\ast, a]$ and hence the local minima of $j$ on $[0,a]$ are at $0, a$ and the local maximum of $j$ on this interval is at $t_\ast$.

We now specify to the case of $f(x) = x^{1/2}$, $F(x)= \frac12x^{-1/2}$, $F'(x) = -\frac{1}{4}x^{-3/2}$, and $c=1$. Because,
$t_\ast$ is where $j'(x)=0$, $t_\ast$ solves
\[\frac{1}{2}x^{-1/2}-\left(-\frac{1}{4}a^{-3/2}+m\right)(x-a)- \frac12a^{-1/2} = 0.\]
Let $u = x^{1/2}$, $q=\frac{1}{4}a^{-3/2}-m$. So,
\[\frac{1}{2u}+q(u^2-a)- \frac12a^{-1/2} = 0,\]
which is
\[qu^3+\left(-\frac12a^{-1/2} - aq\right)u+\frac{1}{2} = 0.\]

We know already that $j' =F-G$ vanishes at $x = a >0$, so the above cubic polynomial in $u$ has $u = a^{1/2}$ as a root. So, we can factor it to obtain
\[0=qu^3+\left(-\frac12a^{-1/2}-qa\right)u+\frac12
=q(u-a^{1/2})\left(u^2+a^{1/2}u-\frac1{2a^{1/2}q}\right).\]
So, applying the quadratic formula to the second term, omitting the negative root because $u > 0$, and using $t_\ast = u^2$ gives:
\begin{equation}
\label{t ast}t_\ast = \left(\frac{-a^{1/2}+\sqrt{a + \frac{4}{2a^{1/2}(\frac14a^{-3/2}-m)}}}{2}\right)^2
=\frac{a}4\left(-1+\sqrt{1 + \frac{8}{1-4ma^{3/2}}}\right)^2.
\end{equation}
Notice that when $m = 0$, this root reproduces the other root that we already had: $x = a$. When $m < 0$, we have $0 < t_\ast < a$.

We also note that 
\[g'(0) = G(0) = \frac{3}{4}a^{-1/2}-ma.\]

So, \[\gamma(1/2) \leq \inf_{a > 0, m \leq 0}\left[ j(t_\ast) - \min(j(0), 0)+g'(0) \right],\]
where we have calculated the relevant terms. We choose $a = 8, m = -0.03314563$ to obtain
$\gamma(1/2) \leq 1.02259.$ This is a great improvement on known results.

Note that we did not have a formula for the optimal values of $a$ and $m$. These specific values were gotten by MATLAB optimization since we had an explicit bound for $E_{x^{1/2}}(1)$. However, prior to this, using the visual graphing calculator software in the link below, we were able to obtain a close approximation to the optimal value even without a closed form for the answer.
\end{example}

\begin{example}\label{f_1 quadratic}
In the last example of this section, we will attempt to do the same calculation for $f(x) = f_1(x) = \frac{x}{x+1}$. An issue that appears with this is that we need to find the parameters $a$ and $m$ for all values of $c$ when there is no formula for $a$ and $m$, as in the previous example.
For this reason, we will somewhat gloss over the details of that last part of the calculation. The reason is that we could rigorously prove such estimates but because we will be using a way of doing that for the estimate in the next section for our main theorem, we choose not to do it here in this toy example that already does not provide the best constant $C$ that this paper presents.

We proceed as in the previous example: $f(x) = \frac{x}{x+1}$, $F(x) = \frac{1}{(x+1)^2}$, $F'(x) = -\frac{2}{(x+1)^3}$ and
\begin{equation*}
G(t) = \left\{\begin{array}{ll} 
\left(-\frac{2}{(a+1)^3}+m\right)(t-a)+ \frac{1}{(a+1)^2},&  0 \leq t < a \\
\frac{1}{(t+1)^2} , & t \geq a
\end{array}\right.
\end{equation*}
and hence
\begin{equation*}
g(x) = \left\{\begin{array}{ll} 
\left(-\frac{1}{(a+1)^3}+\frac{m}{2}\right)(x-a)^2+ \frac{1}{(a+1)^2}(x-a) + 1-\frac{1}{a+1},&  0 \leq x < a \\
1-\frac{1}{x+1} , & x \geq a
\end{array}\right..
\end{equation*}
We again included the shifting so that $g(a) = f(a)$.

In this case, $f'(0)=F(0)=1< \infty$. So, $F-G > 0$ on $(-1, a)$, $F-G \leq0$ on $[t_\ast, a)$ and $F-G = 0$ on $[a, \infty).$ The main difference is that $t_\ast$ may not be in the interval $(0, a)$ to disrupt the monotonicity of $f-g$ on the interval $[0,a]$.

So, with $j(x) = f(x) - g(x)$, $x =t_\ast$ is where
\[0 = j'(x) =  \frac{1}{(x+1)^2} + \left(\frac{2}{(a+1)^3}-m\right)(x-a)- \frac{1}{(a+1)^2},\]
or equivalently:
\[(a+1)^3 + q(x-a)(x+1)^2- (a+1)(x+1)^2 = 0,\]
where $q = 2-m(a+1)^3 > 2$ since $m < 0$. 
We factor the difference of squares then use $x = t_\ast < a$:
\[q(x-a)(x+1)^2+(a+1)[(a+1)+(x+1)][(a+1)- (x+1)] = 0,\]
\[\left[q(x+1)^2-(a+1)(x+a+2)\right](x-a) = 0,\]
\[qx^2+(2q-a-1)x+(q-a^2-3a-2) = 0.\]
The quadratic formula then provides
\begin{align*}
x &= \frac{-(2q-a-1) \pm \sqrt{(2q-a-1)^2-4q(q-a^2-3a-2)}}{2q}\\
&=\frac{-2q+a+1 \pm \sqrt{(a^2+2a+1)+(4qa^2+8qa+4q)}}{2q}\\
&=\frac{-2q+(a+1)\left(1 \pm \sqrt{1+4q}\right)}{2q}
\end{align*}
Because we are interested in the intersection point $x = t_\ast$ that is in $(-1, \infty)$ and since $q > 0$, we obtain
\[t_\ast=-1+\frac{(a+1)\left(1 + \sqrt{9-4m(a+1)^3}\right)}{4-2m(a+1)^3}.\]

The key values for evaluating $\sup_{x \geq 0}j(x) - \inf_{x \geq 0}j(x)$ are $j(0)$, $j(t_\ast)$,  and $j(a) = 0$. It is always the case that $F \leq G$ on $[t_\ast, a]$ so $j(t_\ast)$ is a local maximum. 
Also, if $t_\ast \leq 0$ then $j(x)$ is decreasing on $[0, a)$. So, 
\[\sup_{x \geq 0}j(x) - \inf_{x \geq 0}j(x) = 
\left\{\begin{array}{ll}
j(t_\ast)-\min(j(0), 0), & t_\ast > 0 \\
j(0), & t_\ast \leq 0
\end{array}\right..
\]

We then note that $g'(0) = a\left(\frac{2}{(a+1)^3}-m\right)+ \frac{1}{(a+1)^2}$ to obtain an explicit expression for the right-hand side of 
\begin{equation}\label{f1 piecewise quadratic inf}
\frac{E_{f_1}(c)}{c/(c+1)} \leq \frac{\inf_{a > 0, m<0}\left(\sup_{x \geq 0}j(x) - \inf_{x \geq 0}j(x)+cg'(0)\right)}{c/(c+1)}.
\end{equation}
Using MATLAB's \texttt{patternsearch()} minimization function, we generated the values of this bound for $c = 0.01, 0.02, \dots, 15$ 
and obtained a maximum value of $1.07688$, which roughly occurs in the range of $[.2, .4]$. The values of $c$ that we used in this calculation are denoted by $c_k$ and they satisfy $c_{k+1}-c_k = 0.01$.
See Figure \ref{piecewiseQuadraticMin_graph} for the graph of the computed values using the smaller fixed spacing of  $c_{k+1}-c_k = 0.001$. 

Notice that although this method does not exactly calculate the infimum in (\ref{f1 piecewise quadratic inf}), we are only aimed for an upper bound for $E_{f_1}(c)$ anyway so there is no issue with generating values of $a$ and $m$ that may or may not be precise estimates of the actual optimal choice of the $a, m$ for a given value of $c$.
\begin{figure}
     \centering
     \begin{subfigure}[b]{0.3\textwidth}
         \centering
         \includegraphics[width=\textwidth]{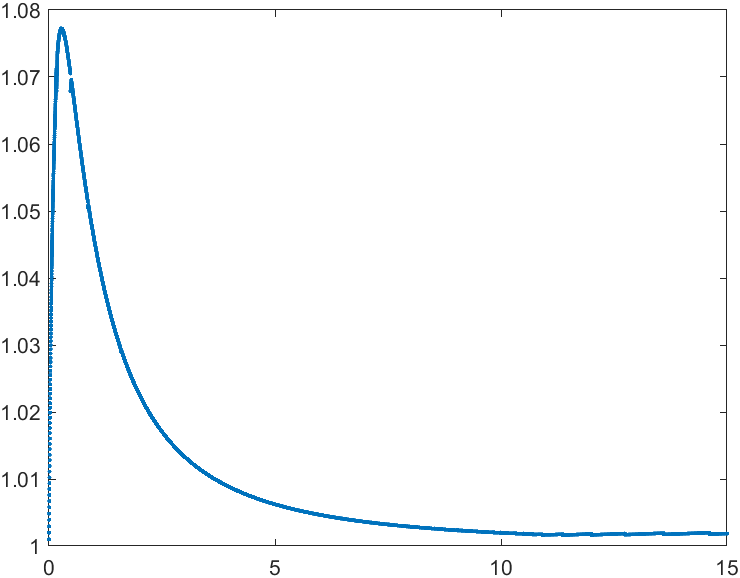}
         \caption{full graph} \label{piecewiseQuadraticMin_graph1_ext}
     \end{subfigure}
     \hspace{1in}
     \begin{subfigure}[b]{0.3\textwidth}
         \centering
         \includegraphics[width=\textwidth]{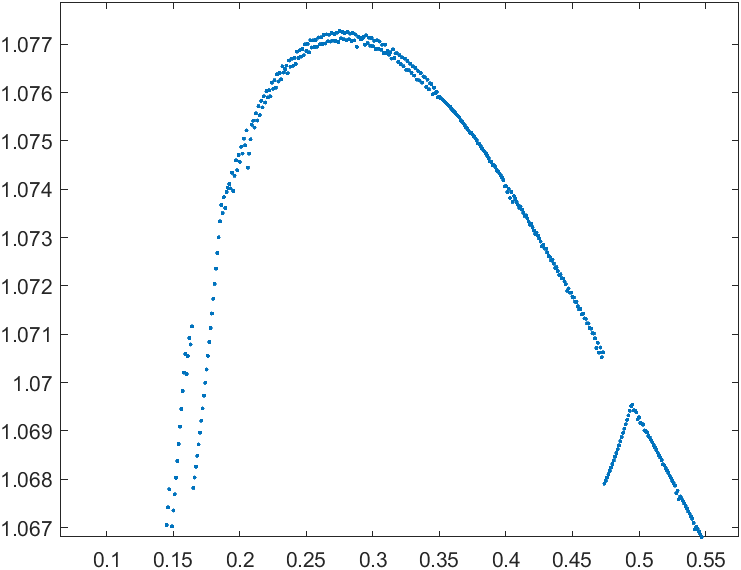}
         \caption{Graph nearby maximum}
         \label{piecewiseQuadraticMin_graph2_ext2}
     \end{subfigure}
        \caption{Plot of computed constants $C_k$ verses $c_k$ with $c_{k+1}=c_k+0.001$. One can note the discontinuities potentially due to the minimization procedure used.}        \label{piecewiseQuadraticMin_graph}
\end{figure}

We now need a way to piece these together to get a bound for all $c$. First, note that for $c \leq 0.01$ and $c \geq 15$, we can use the trivial bound of $E_{f_1}(c) \leq \min(c,1)$ from Example \ref{simple example} to obtain
\[\frac{E_{f_1}(c)}{c/(c+1)}  \leq \max\left(0.01+1, \frac{16}{15}\right) < 1.0667.\]
Then inside the interval $(0.01, 15)$ we will use Lemma \ref{C continuity points} to see that (\ref{C continuity - D interval}) holds for 
\[D < 1.07688\left(\frac{.01+.01+1}{.01+1}\right) \leq 1.08755.\]
If one instead follows the advice of Remark \ref{D_k remark}, we use a MATLAB calculation with the computed values of $C_k$ to obtain the constant
\[\max_k D_k < 1.08536.\]

If one uses values of $c_k$ the smaller spacing of $0.001$, then the maximum value of the $C_k$ increases to $1.07727$ and $\max_k D_k$ decreases to $1.07812$. See Figure \ref{piecewiseQuadraticMin_graph} for a graph of the values $C_k$ computed for these $c_k$.
So, the bound for the constant $C$ in (\ref{CI}) that we obtain for all $c$ is $C < 1.07812$. 

Notice that the distance that this estimate for $\gamma(1/2)\leq C < 1.07812$ is from $1$ is considerably larger than the estimate than we got in Example \ref{x^1/2 piecewise quadratic} of $\gamma(1/2) < 1.02259$ by using the same approximation method directly to $f(x) = x^{1/2}$. However, they are both much closer to $1$ than the previously known estimates. 

The main theorem proved in the next section provides the best estimate of this paper and we note here that the choice of parameterized approximation function $g(x)$ in that result is not possible for $f(x) = x^{1/2}$ because $g$ is bounded whereas $f$ is not.
\end{example}

\section{Proof of Main Theorem}\label{Proof of Main Theorem}

In this section we prove the main theorem, Theorem \ref{Main Thm}. The proof bears some resemblance to Example \ref{f_1 quadratic}, however due to the lack of a closed form solution it is more tricky, both mathematically and numerically. 

We begin by letting $f(x) = \frac{x}{x+1}$ and $G(x)=g'(x) = ae^{-bx^2}$, a Gaussian with two parameters $a, b \geq 0$. We have that $G$ is the Fourier transform of a positive function in $L^1$ so $\vertiii{g}_{\CL(\R)} = G(0) = a$. This is the easiest part of the estimate of $E_{f}(c)$.

We will not be able to write down a closed form solution for $\sup_{x \geq 0} j(x) - \inf_{x \geq 0} j(x)$ so we will lay out the calculations that are used in a MATLAB function for calculating it. Note that because it involves calculating all the local maxima/minima, we need to express the calculation in a way that identifies all these extreme points and gives correct bounds without knowing their exact locations. We do this by studying the nature of the extreme points, using the MATLAB function \texttt{fzero()} to identify approximate roots of $j'$, numerically verifying that the roots of $j'$ are nearby these approximate roots within some numerical tolerance, and then using this to provide a numerical upper bound for $\sup_{x \geq 0} j(x) - \inf_{x \geq 0} j(x)$.

Let 
\[g(x) = \int_0^x G(t)dt =  \frac{a}{2}\sqrt{\frac{\pi}{b}}\erf(\sqrt{b}x)\]
and $j(x) = f(x) - g(x).$ The interior local extrema of $j(x)$ are at $x$ where $\frac{1}{(x+1)^2} = ae^{-bx^2}$, which is
\[bx^2-2\log(x+1)-\log(a) = 0.\]

Define $\Phi(x) = bx^2-2\log(x+1)-\log(a)$ on the domain $(-1, \infty)$. Note that $\Phi(x)$ has the same sign as $j'(x)$. We wish to explore the nature of the roots of $\Phi$, particularly those in $[0, \infty)$. So, let us explore the behavior of $\Phi$ on its domain $(-1, \infty)$. If $x \to -1^+$ or $x \to \infty$, $\Phi(x) \to \infty$. Because $\Phi'(x) = 2bx-\frac{2}{x+1}$, solving $\Phi'(x) = 0$ is equivalent to solving $x^2+x-1/b = 0$. So, $\Phi'(x) = 0$ at
\[x_\ast = \frac{-1 + \sqrt{1 + \frac{4}{b}}}{2} > 0,\]
where we omitted the solution to the quadratic that is outside the domain of $\Phi$. The value $x_\ast>0$ is hence the global minimum of $\Phi$. Roughly speaking, the graph of $\Phi$ looks like an upward facing parabola with vertex at $x = x_\ast$ but with the line $x = -1$ as a vertical asymptote.

We now want to specify an interval in which we are guaranteed to find all the roots of $j'$ in $(0, \infty)$. We do this by first finding the right end point $x_e$ of such an interval. We  find a value of $x_e$ so that $\Phi(x) > 0$ for all $x \geq x_e$. Recall that $-\log(x+1) > -x$ for $x > 0$. So,
\[\Phi(x) > bx^2-2x-\log(a).\] Hence $\Phi(x) > 0$ for $x \geq x_e$ where
\[x_e = \frac{1+\sqrt{1+b\log(a)}}{b} > 0.\]

Now, suppose for the moment that $\Phi \geq 0$ on $[0, \infty)$. Then $F \geq G$. We already explored the case that $F \geq G$ for $x^r$ in Example \ref{F geq G example}. We argued rather generally that the optimal choice of $G$ under this condition should equal $F$ on some 
interval $[t_\ast, \infty)$ and should be a line on $[0, t_\ast]$. 
Something similar holds to that of Example \ref{f_1 quadratic}, but the result gotten using $F \geq G$ would give a worse estimate because of the lack of the additional parameter $m$ which would be fixed to equal zero. 

So, the assumption that we make is to require that $\Phi$ not be entirely non-negative on $[0, \infty)$ which implies that $\Phi$ has a root in $(0, \infty)$. Because of what we already said, it suffices to assume that $\Phi(x_\ast) < 0$. Because $\Phi''(x) = 2b + \frac{2}{(x+1)^2} > 0$, there are exactly two roots in the domain of $\Phi$, one of which belongs to $(-1, x_\ast)$ and one which belongs to $(x_\ast, x_e)$.

However, the assumption that $\Phi(x_\ast) < 0$ does not guaranteed that there is a root in $(0, x_\ast)$. Because $\Phi$ is strictly decreasing on $(-1, x_\ast)$ and $\Phi(x_\ast) < 0$, there is a root in $(0, x_\ast)$ if and only if $\Phi(0) = -\log(a) > 0$ which is true if and only if $a < 1$. If $a \geq 1$, then $\Phi'(x) \leq 0$ on $[0, x_\ast]$ so $j$ will have a local maximum at $0$ within its domain $[0, \infty)$. Otherwise, there is a local minimum at $0$.

So, we have identified the general regions where the possible local extrema are to be found. We also need to account for the limit as $x \to \infty$ because $\lim_{x\to\infty}j(x)$ might not be zero as it was in all the examples of Section \ref{Examples} that used this method. Because $j'(x)$ is positive for large $x$ and $\lim_{x\to\infty}\erf(x) = 1$, this contributes the term $j_\infty = 1 - \frac{a}{2}\sqrt{\frac{\pi}{b}}$ as a possible value for the supremum of $j$.

In summary, with the assumption that $\Phi(x_\ast) < 0$, there is a root $x_2$ of $j'$ in $(x_\ast, x_e)$ and a root $x_1$ in $(-1, x_\ast)$ which belongs to $(0, x_\ast)$ if and only if $a < 1$. So, because $\{x: j'(x) < 0\} = (x_1, x_2)$, we have
\begin{equation}\label{j oformula - main theorem}  \sup_{x \geq 0}j(x) - \inf_{x \geq 0}j(x) = \left\{\begin{array}{ll}
\max(j(x_1), j_\infty) - \min(j(0), j(x_2))   , & a < 1\\
\max(j(0),j_\infty) - j(x_2)  , & a \geq 1 
\end{array} \right.. 
\end{equation}

We now proceed to explaining how to programatically estimate \begin{equation}\label{j c exp}
\frac{\sup_{x \geq 0}j(x) - \inf_{x \geq 0}j(x)}{c/(c+1)}
\end{equation}
using Equation (\ref{j oformula - main theorem}). This will be encoded in our MATLAB function \texttt{ErfMin()} which is a function of the variables $c, a, b$. The only technicalities are in estimating $j(x_1)$ and $j(x_2)$ because we do not have a formula for $x_1$ or for $x_2$. 

There are two cases to consider. The first case is when $\Phi(x_\ast) \geq 0$. 
Because we have already dismissed this by us thinking that it will produce estimates worse than the explicit estimates in the prior section, we make \texttt{ErfMin()} produce the value $10$ in the case that $\Phi(x_\ast) \geq 0$. 
This is considerable larger than the values of $C$ that we already obtained in prior sections. If all the values produced by \texttt{ErfMin()} are less than $10$ for certain values of $c, a, b$ then the case of $\Phi(x_\ast) \geq 0$ never occurred. This is exactly what we obtain with the numerical calculations discussed later in the proof so we may proceed to assume that $\Phi(x_\ast) < 0$ which implies that Equation (\ref{j oformula - main theorem}) is valid.

We are then guaranteed a root $x_2$ in $(x_\ast, x_e)$. So, we use the MATLAB function \texttt{fzero()} applied to $\Phi$ over $(x_\ast, x_e)$ to produce an approximate root using the fact that $\Phi$  changes signs from $x_\ast$ to $x_e$, as is required by \texttt{fzero()}. Let $\tilde{x}_2$ denote this approximate root which is encoded as \texttt{x\_2}. Likewise, in the case of $a < 1$, $\tilde{x}_1$ will denote the root in $(0, x_\ast)$ gotten by \texttt{fzero()} and will be encoded as \texttt{x\_1}.

We incorporate into the function \texttt{ErfMin()} how close the value of $j(\tilde{x}_i)$ may be to $j(x_i)$ as follows. Let $T = 10^{-5}$ be the tolerance in the distance from the roots of $j'$ and their respective approximate roots. Let $T_{f} = 10^{-10}$ denote the computation tolerance. 
The function \texttt{ErfMin()} produces an error if it is not the case that $j'(\tilde{x}_2 - T) \leq -T_f$ and $j'(\tilde{x}_2 + T) \geq T_f$. This provides some assurance within an error of $T_f$ that $j'$ has a root in $[\tilde{x}_2-T, \tilde{x}_2 + T]$ because $j'$ should change signs from negative to positive on this interval if it contains $x_2$. It also produces an error if $\tilde{x}_2- T$ is not in the set $[T_f, \infty)$.

Likewise, in the case that $a< 1$, the function \texttt{ErfMin()} produces an error if it is not the case that $j'(\tilde{x}_1 - T) \geq T_f$ and $j'(\tilde{x}_1 + T) \leq -T_f$. It also produces an error if $\tilde{x}_1- T$ is not in the set $[T_f, \infty)$. 

This provides relative assurance that the critical points of $j$ are indeed where they are computed by \texttt{fzero()} within a distance of $T$ with numerical tolerance of $T_f$ and that these tolerance intervals are in the domain within a tolerance of $T_f$.

Knowing that no grave error was produced in the use of \texttt{fzero()}, we now bound how much compensation for error we must incorporate into \texttt{ErfMin()} for it to produce a bound for (\ref{j c exp}). We know that $|\tilde{x}_i-x_i| \leq T$ so $|j(\tilde{x}_i)-j(x_i)| \leq T\sup_{x \geq 0} |j'(x)|$.
Note that $|j'(x)| \leq 1 + a$ for all $x \geq 0$. So, for the potentially two roots calculated in \texttt{ErfMin()}, we incorporate the additional error term of $2(1+a)T$ encoded as \texttt{q\_error}. 

Now, given arguments $c, a, b$, \texttt{ErfMin()} produces an upper bound for (\ref{j c exp}) wand hence an upper bound for $\frac{E_{f_1}(c)}{c/(c+1)}.$
Let $c_k, a_k, b_k$ be arguments provided to the function \texttt{ErfMin()} where $0 < c_1 < \cdots < c_n$ and denote $C_k$ be the respective values produced. Then we apply Lemma \ref{C continuity points} to obtain a constant $C = \max_k D_k$ so that (\ref{f1 comm ineq}) holds for every value of $c=\vertiii{[A,X]}$ in the interval $[c_1, c_n]$. The values of $a_k, b_k$ are gotten through an optimization procedure using MATLAB's \texttt{patternsearch()} function. Because we provide  \texttt{.txt} files containing the product of this procedure, it is unnecessary to go into exactly how this was done.

So, we apply this with
\[c = .0195, .0195+.0005, \dots, 1.5, 1.5+.005, \dots, 10, 10+.05, \dots, 40\] 
and $a_k, b_k$ from files \texttt{EM\_as.txt} and \texttt{EM\_bs.txt}.
This produces a constant $C$ for (\ref{j c exp}) that is less than $1.01975$ for all $c$ in $[0.0195, 40]$.
\begin{figure}
     \centering
     \begin{subfigure}[b]{0.3\textwidth}
         \centering
         \includegraphics[width=\textwidth]{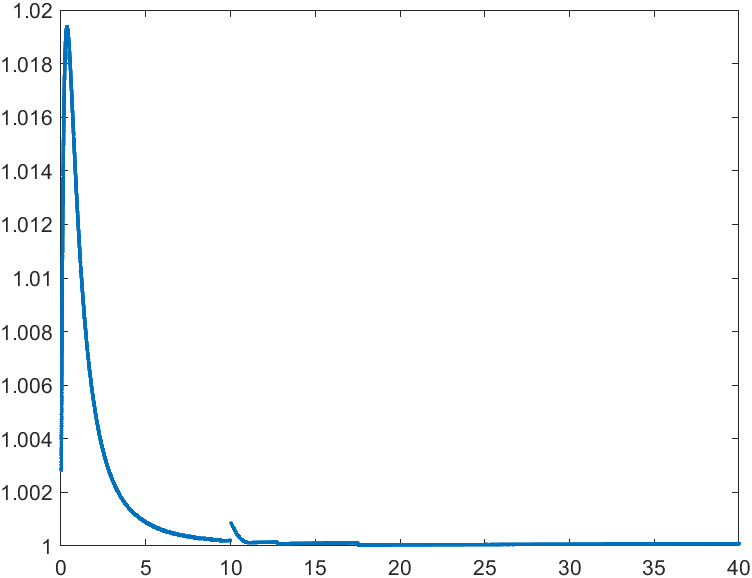}
         \caption{full graph} \label{ErfMin_graph1_updated}
     \end{subfigure}
     \hfill
     \begin{subfigure}[b]{0.3\textwidth}
         \centering
         \includegraphics[width=\textwidth]{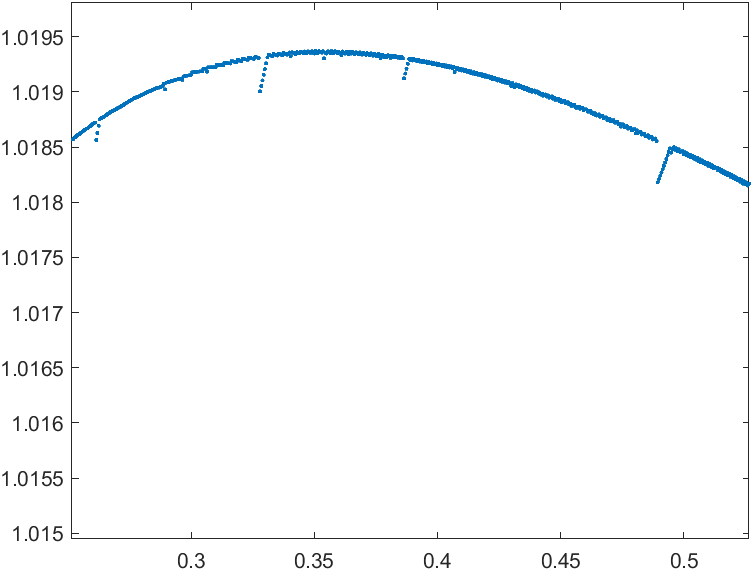}
         \caption{Graph nearby maximum}
         \label{ErfMin_graph2_updated}
     \end{subfigure}
     \hfill
     \begin{subfigure}[b]{0.3\textwidth}
         \centering
         \includegraphics[width=\textwidth]{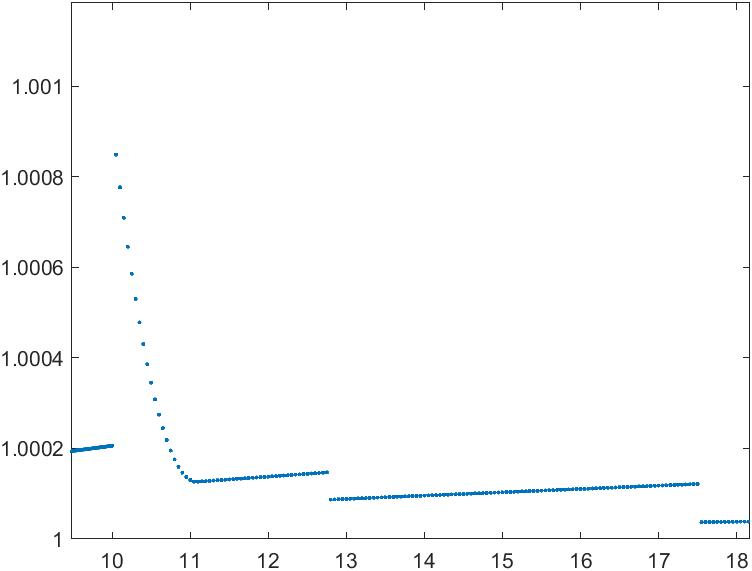}
         \caption{Graph near discontinuities with small values}
         \label{ErfMin_graph3_updated}
     \end{subfigure}
        \caption{Graph of the values of $C_k$ generated by \texttt{ErfMin()} vs. $c_k$ as discussed in the proof of the main theorem. Discontinuities suggest that the $a_k, b_k$ used are perhaps not optimal so the constant generated by this method may be smaller than what was computed using our values of $a_k, b_k$.}        \label{ErfMin Graphs}
\end{figure}

For the interval $(0, 0.0195)$, we apply the Lipschitz constant bound $c$ to obtain a constant of
\[\sup_{c \in (0, 0.0195)} \frac{c}{c/(c+1)} = 1.0195.\]
For the interval $(40, \infty)$, we apply the  bound $1-\frac{1}{4c}$ for $c \geq 1/2$ in Example \ref{x/(x+1) shift} to obtain a constant of
\[\sup_{c \in (40, \infty)} \frac{1-\frac{1}{4c}}{c/(c+1)} < 1.0186 .\]
So, we obtain $C < 1.01975$ applying for all $c > 0$. This concludes the proof.

\section{Further Improvements in Special Cases}\label{Special Cases}

Although we are aimed at the constant $C$ in (\ref{f1 comm ineq}), in this section we pause to dig into some of the finer detail to extract more from what we have already discussed for specific operator monotone functions.
We will do this by estimate the value of the constant $C(c)$, which is defined to be the optimal constant for
\[\vertiii{[f_1(A),X]} \leq C(c) f_1\left(c\right)\]
for all $A, B \in B(\H)$ satisfying $A \geq 0$, $\vertiii{X} \leq 1$, and $\vertiii{[A,X]}=c$. 
Conjecture \ref{Conjecture 1} can be restated as saying that $C(c)$ is identically equal to $1$. We will generically use $C(c)$ to mean to apply to all unitarily-invariant norms on a finite dimensional $B(\H)$ or the operator norm on $B(\H)$ where $\H$ is more general. If we wish to specify a specific norm, then we will state so specifically.

We now perform (\ref{f_t reduction}) with more care.
First suppose that $f \geq 0$ is operator monotone on $[0, \infty)$ with $\alpha, \beta = 0$ in (\ref{integral representation}) for simplicity. Let $c = \vertiii{[A,X]}$. Then by (\ref{f_t reduction}) and (\ref{f_t scaling}),  
\begin{align}\nonumber
\vertiii{[f(A),X]} &\leq \int_0^\infty t\vertiii{[f_1(t^{-1}A),X]}d\nu(t) \\
&\leq \int_0^\infty C\left(\frac{c}{t}\right)tf_1\left(\frac{c}{t}\right)d\nu(t) = \int_0^\infty C\left(\frac{c}{t}\right)tf_t(c)d\nu(t). 
\label{f_t finer reduction}
\end{align}
Moreover, from Example \ref{simple example}, we see that $C(t)$ will satisfy $C(t) \to 1$ as $t \to 0^+$ or $t\to \infty$ because
\begin{equation}
\label{simple bound}
1 \leq C(t) \leq \frac{\min(t,1)}{t/(t+1)} = \min\left(t+1, \frac{t+1}{t}\right).
\end{equation}
So, if $C(t)$ is not identically $1$ and the support of $\nu$ has non-empty intersection with every neighborhood of $\{0, \infty\}$, then $C(t) < C$ on a set with non-zero $\nu$-measure so we will obtain 
\[\vertiii{[f(A),X]} < C\int_0^\infty tf_t(c)d\nu(t) = Cf(\vertiii{[A,X]})\]
whenever $[A,X] \neq 0$.
If either $\alpha$ or $\beta$ are positive then the strict inequality applies without any assumption on the support of $\nu$.

We now explore how this can provide estimates of the form $\gamma(r) < C$  in the following examples.

\begin{example} Consider the bound in (\ref{simple bound}). If we apply this with $f(x) = x^{1/2}$ with the integral representation in (\ref{sqrt int rep}) and $c = 1$, we have $A^{1/2} = \frac2\pi\int_0^\infty f_{t^2}(A)dt$ so
\[
\vertiii{[f(A),X]} \leq \int_0^\infty  \min\left(\frac{1}{t}, 1\right)f_{t^2}(1)dt. 
\]
Calculating this integral gives exactly the bound for $\gamma(1/2)$ gotten by Boyadzhiev as discussed in Example \ref{Boyadzhiev example}. In fact, it is no surprise because (\ref{simple bound}) is merely the commutator inequalities performed by him in two cases. 
\end{example}

\begin{example}
In this case, we will make use of (\ref{scaled}) to obtain a bound for $C(c)$ for the operator norm. Applying this equality for some $t > 0$ to $f_1$ with $\|X\| \leq 1$ and $c = \vertii{[A,X]}$ provides
\[\vertii{[f_1(A),X]} \leq \frac{1+t^2}{2t} \frac{f_1(2tc)}{c/(c+1)}f_1(c) = \frac{(1+t^2)(c+1)}{2tc+1}f_1(c).\]
For this value of $c$, we choose the optimal value of 
\[t = \frac{-1+\sqrt{1+4c^2}}{2c} > 0\]
which after some rearranging gives
\[\vertii{[f_1(A),X]} \leq \frac{c}{\frac{1}{2}+\sqrt{\frac{1}{4}+c^{2}}}\]
so
\[C(c) \leq \frac{c+1}{\frac{1}{2}+\sqrt{\frac{1}{4}+c^{2}}}.\]

Notice that this bound has a maximum value of $5/4$ at $c = 2/3$. So, although this upper bound is larger than $\csc(1)$, it is smaller than this value outside the interval $[.268, 1.701]$. So, in the case of the operator norm, we obtain the best explicit bound:
\[C(c) \leq  \min\left(\csc(1), \frac{c+1}{\frac{1}{2}+\sqrt{\frac{1}{4}+c^{2}}}\right).\]

Taking this together with (\ref{f_t finer reduction}), we obtain the bound of $\gamma_0(1/2) < 1.102$. Note that this estimate is better than the previously known estimate of $\pi/\sqrt2 \approx 1.1284$. So, in a certain sense better estimates were very close from being obtained. 
\end{example}

Note that there exist sharper estimates, but we are not aware of any that can be written explicitly. For instance, we could use the estimate that we obtained in Theorem \ref{sin(t) estimate}. However for a given value of $c$, finding its optimal value of $t$ in such a scaled inequality requires solving $\tan(t)=t(tc+1)$.

We now apply this method along with the data from the main theorem to prove Theorem \ref{Main Thm cor sqrt}. This is a great improvement of all the known results. Moreover, it is considerably smaller than the bound of $\gamma(1/2) < 1.01975$ from the main theorem.

\begin{proof}
We first begin by using the integral representation
\[A^{1/2} = \frac1\pi\int_0^\infty f_t(A)t^{-1/2}dt\]
from (\cite{BarrySimon - Loewner}, Chapter 4). The appearance of $f_t$ rather than $f_{t^2}$ is why we prefer this representation.
So, if $\vertiii{X} \leq 1$ and $\vertiii{[A,X]}=1$ then with using a change of variables
\[\vertiii{[A^{1/2},X]} \leq \frac1\pi\int_0^\infty C\left(\frac{1}{t}\right)\frac{f_t(1)}{t^{1/2}}dt = \frac1\pi\int_0^\infty \frac{C(1/t)}{(1+t)t^{1/2}}dt= \frac1\pi\int_0^\infty \frac{C(t)}{(1+t)t^{1/2}}dt.\]

We estimate this integral in three pieces, as before. On the first and last interval we use (\ref{simple bound}). In particular, we use $C(t)\leq 1+t$ on $(0, .0195]$ and $C(t)\leq(1+t)/t$ on $[40, \infty)$.
From the data calculated in the proof of the main theorem, there are $c_k$ with $c_1 = .0195 < c_2 < \cdots < c_n=40$ with bounds $C_k$ for the commutator inequality with $c = c_k$. By, Lemma \ref{C continuity}, for a value of $t$ in the interval $[c_k, c_{k+1}]$, we have \[C(t) \leq C_k\left(\frac{t+1}{c_k+1}\right).\]
So, 
\begin{align*}
\vertiii{[A^{1/2},X]}  &\leq \frac1\pi\left(\int_0^{0.0195}t^{-1/2}dt+\sum_{k=1}^{n-1}\frac{C_k}{c_k+1}\int_{c_k}^{c_{k+1}} t^{-1/2}dt + \int_{40}^\infty t^{-3/2}dt\right)\\
&=\frac1\pi\left(2( 0.0195^{1/2})+\sum_{k=1}^{n-1}\frac{2C_k}{c_k+1}\left(c_{k+1}^{1/2}-c_k^{1/2}\right) + 2(40^{-1/2})\right) < 1.00891.
\end{align*}
\end{proof}
Note that in all the results of this section, they provide smaller estimates for certain functions $f$. However, as long as $C(c)$ is not shown to be $1$ almost everywhere, this method cannot be used to prove $\gamma(r) = 1$.

\section{False Positives?}\label{False Positives?}

In this section we discuss two results that are a solution to the desired inequality (\ref{CIneqMod0}) with $C = 1$ in some special cases. On their face, these results suggest that Conjecture \ref{Conjecture 1} could plausibly be true. However, we prove in this section that these results actually are still true under much weaker conditions than $f$ being operator monotone, such as $f$ merely being non-negative and strictly concave. The second result we discuss of Loring and Vides generalizes and is also interesting for its own sake.

It is common that some non-trivial operator condition be required beyond mere smoothness for inequalities of numbers to extend to operators. Like in the proof of Ando's theorem, this is used in the proof of the associated norm inequalities. So, the results of this section suggest that these results could unfortunately be dead-ends as far as the journey to obtain a proof of the desired result. We also provide an example of Ando's inequality failing (and hence the conjectured inequality as well) for a smooth non-negative concave function.

It is the author's view that the compelling evidence in the literature that Conjecture \ref{Conjecture 1} for the operator norm holds for $f(x)=x^{1/2}$ comes from the simulations alluded to by Pedersen and the Monte Carlo simulations done by Loring and Vides in \cite{Loring Vides}. Also, as discussed in the introduction, our paper further supports the belief that the conjecture may be true for operator monotone functions. Possible future Monte Carlo simulations for the function $f(x)=x/(x+1)$ might support the belief that the desired inequality holds for operator monotone functions. 

\vspace{0.1in}

Consider the following counter-example to Ando's theorem (and hence Conjectures \ref{Conjecture 0}, \ref{Conjecture 1}, \ref{Conjecture 2}, and \ref{Conjecture 3}) for the operator norm.
\begin{example}\label{counter-example}
Let 
\[f(x) = \left\{ \begin{array} {ll}
x, & x \leq 1 \\
1, & x > 1
\end{array}\right..\]
Note that $f(x)\leq x$ for all $x \geq 0$.
The function $f$ is evidently concave and non-negative on $[0,\infty)$. However, this function is not operator Lipschitz on this interval. In particular, $f(x) = \frac12\left(x+1-\left|x-1\right|\right)$ and the absolute value function is not operator Lipschitz (see \cite{Kato abs function Lipschitz}). Moreover, because the inequality
\[\||A|-|B|\| \leq C\|A-B\|\]
is homogeneous, we can assume that counter-examples to this can be chosen with spectrum in $[-1,1]$. This implies that $f$ is not operator Lipschitz on $[0, 2]$.

However, if Ando's inequality were to hold for $f$, even with a constant $C \geq 1$, then
\[\|f(A) - f(B)\| \leq Cf(\|A-B\|) \leq C\|A-B\|\]
which implies that $f$ would be operator Lipschitz with this constant, which it is not.

We now show that Ando's inequality does not apply with any single constant $C \geq 1$ for all non-negative, smooth, strictly-concave functions on $[0,\infty)$.
Note that we can always uniformly approximate $f$ by a non-negative, smooth, concave function by smoothening out a piecewise linear approximation of $f$. Then it can be made strictly concave by adding $\varepsilon \frac{x}{x+1}$ for $\varepsilon>0$ small. Let $f^\varepsilon$ be this approximation. Then we would have
\[\|f(A)-f(B)\|= \lim_{\varepsilon \to 0^+}\|f^\varepsilon(A)-f^\varepsilon(B)\| \leq \lim_{\varepsilon \to 0^+}Cf^\varepsilon(\|A-B\|) = Cf(\|A-B\|)\]
which we already showed was not true.

This shows that Ando's inequality with a constant $C \geq 1$ fails in the general case of the operator norm for non-negative, smooth, strictly-concave functions on $[0,\infty)$. However, it does not make any statement about this inequality expressed in the form of that of Corollary \ref{Ando Norms of differences of f} for other unitarily-invariant norms. It is an easy consequence that it does not hold with constant $C=1$ for Schatten-$p$ norms for $p$ large. However, it does hold for the Hilbert-Schmidt norm, as we discuss at the end of this section.
\end{example}

We now discuss the first example of a possible false positive to the conjecture.
\begin{example}
In (Remark 4, \cite{BK-commutators}), the authors of that paper proved Conjecture \ref{Conjecture 0} in the case of $2 \times 2$ matrices when $A=B$. However, with little effort their proof actually shows that the inequality holds for \emph{all} non-negative concave functions.

We will discuss this in detail now by following the exposition in \cite{BK-commutators}. Without loss of generality because the norm $\vertiii{-}$ is unitarily invariant, we may assume that $A = \bp a_1 & 0 \\ 0 & a_2\ep$ is diagonal with $a_2 \geq a_1 \geq 0$. Write $X = \bp x_{11} & x_{12} \\ x_{21} & x_{22}\ep$. Commutators of $X$ with $A$ or with $f(A)$ are the same as commutators of $X - \diag(X) =   \bp 0 & x_{12} \\ x_{21} & 0\ep$, respectively. Note that $\|X\|\leq 1$ implies that $|x_{12}|, |x_{21}| \leq 1$.

Let $f \geq 0$ be concave on $[0, \infty)$. Note that this implies that $f$ is monotonically increasing. Because
\[[A,X] = 
\bp 0 & (a_1-a_2)x_{12}\\ 
(a_2-a_1)x_{21} & 0 \ep,
\] we see that the singular values of $[A,X]$ are $(a_2-a_1)|x_{12}|$ and $(a_2-a_1)|x_{21}|$. So, the singular values of $f(|[A,X]|)$ are $f\left((a_2-a_1)|x_{12}|\right)$ and $f\left((a_2-a_1)|x_{21}|\right)$.
Similarly, the singular values of $[f(A),X]$ are $(f(a_2)-f(a_1))|x_{12}|$ and $(f(a_2)-f(a_1))|x_{21}|$. 

So, the inequality 
(\ref{BK Ineq}) with $A = B$ and constant $C = 1$ holds for all unitarily-invariant norms on $M_2(\C)$ if and only if for all $t \in [0, 1]$,
\begin{equation}\label{2by2 equivalent}
t(f(a_2)-f(a_1)) \leq f(t(a_2-a_1)).
\end{equation}
This is implied by the non-negativity and concavity of $f$ as follows. $f$ is sub-additive so because $a_2 - a_1 \geq 0$,
\[tf(a_2) \leq tf(a_2-a_1) + tf(a_1).\]
Hence, we only need to show that
$tf(x) \leq f(tx)$ for $x = a_2-a_1\geq 0$. This follows from the standard argument:
\[tf(x) \leq tf(x) + (1-t)f(0) \leq f(tx + (1-t)0) = f(tx).\]

Also, (\ref{2by2 equivalent}) for all $a_2 \geq a_1 \geq 0$ and $t \in [0,1]$ is just the $1 \times 1$ version of the inequality  (\ref{BK Ineq}) with $A = a_1$, $B = a_2$, and $X = te^{i\theta}$.
We here conclude our remark that $f$ was only needed to be non-negative and concave in this example.  
\end{example}

The following result of Loring and Vides is the other result in the literature that we wish to mention in this section.
They proved this by making use of the argument we discussed in Example \ref{Loring Vides example} with $g$ being a line. This result roughly says that we get the conjectured optimal estimate for $x^{1/2}$ if the commutator is not too small. 
\begin{prop} (Lemma 5.1 of \cite{Loring Vides}) \label{Loring Vides}
Suppose that $A, X \in M_n(\C)$ with $A \geq 0$ and $\|X\| \leq 1$. 

If we restrict to $\|A\| \leq 1$ and $\|[A,X]\| \geq 1/4$ then
\[\|[A^{1/2},X]\| \leq \|[A,X]\|^{1/2}.\]
\end{prop}

We now extend Proposition \ref{Loring Vides} as follows. 
\begin{thm}\label{LV extension}
Suppose that $A, X$ are operators in $B(\H)$ with $\H$ separable, $A \geq 0$, and $\vertiii{X} \leq 1$. 
Let $f\in C^1((0, \infty))$ be a non-negative, strictly concave function with $f(0+) = \lim_{x \to 0^+}f(x)=0$. Define the strictly decreasing functions $k(x) = \frac{f(x)}{x}$ and $F(x) = f'(x)$. Then 
\[ \vertiii{[f(A), X]} \leq f(\vertiii{[A,X]})\]
whenever \[(F^{-1}\circ k)( \vertiii{A} ) \leq \vertiii{[A,X]} \leq  \vertiii{A}.\]

Note that $(F^{-1}\circ k)(x) < x$ for every $x \in (0, \infty)$.
\\In the case of $f(x) = f_t(x) =  \frac{x}{x+t}$, 
\[(F^{-1}\circ k)(x)=  t^{1/2}(x+t)^{1/2}-t.\]
\\In the case that $f(x) = x^r$ for $r \in (0, 1)$, 
\[(F^{-1}\circ k)(x)= r^{1/(1-r)}x.\]
\end{thm}
\begin{proof}
By (\ref{Kittaneh ineq}), we need only prove this result for the subset of functions $f$ that uniformly approximate all the types of functions considered.
So, by smoothening out a piecewise linear approximation of $f$, we may assume also that $f$ is smooth.
Then by adding a term of $\varepsilon \frac{x}{x+1}$ for $\varepsilon 
> 0$ small, we may assume also that $f$ is strictly concave.

Because $F=f'$ is strictly decreasing, its range is $\mathcal I = (f'(\infty), f'(0+)),$ where $f'(\infty) = \lim_{x \to \infty}f'(x)$.  Let $m \in \mathcal I$.

We define the function $g(x) = mx$. Let $j = f-g$. Then $j \geq 0$ on the interval $I = [0, x_1]$, where $x_1 > 0$ satisfies $f(x_1) = g(x_1)$ which exists because $f$ is strictly concave. Note that $j(0) = 0$.
Because $j'$ is strictly decreasing, the maximum of $j$ occurs when $F(x) = m$, which is at $x_\ast = F^{-1}(m)$, a point in the interior of $I$. Let $c$ belong to the interior of $I$. So,
\[\sup_{x \in I} j(x) - \inf_{x \in I}j(x) + cg'(0) = j(x_\ast)+mc=
f(x_\ast) - mx_\ast+mc.\]

If $F$ were differentiable, then this expression has a critical point as a function of $m$ when
\[f'(F^{-1}(m))(F^{-1})'(m)  - F^{-1}(m)
 - m (F^{-1})'(m)  = -c,\]
which simplifies to $F^{-1}(m) = c$. So, regardless of whether $F$ is differentiable, we just define $m = F(c) \in \mathcal I$. This implies that $F^{-1}(m) = c$ and $x_\ast = c$.
With this choice of $m$, we have
\[\sup_{x \in I} j(x) - \inf_{x \in I}j(x) + cg'(0) =
f(c).\]

Define $c = \vertiii{[A,X]}$. If $c=0$, then the desired inequality in the statement of the theorem is trivial. If the inequality is proved for all $c \in (0, \vertii{A})$ then by perturbing $f$ by scaling its argument, the result follows for $c = \vertii{A}$.
So, we can suppose that $c$ as defined belongs to the interior of $I$. Now, using the reasoning that went into proving (\ref{fund comm ineq}), we see that if the spectrum of $A$ belongs to $I$ then
\[
\vertiii{[f(A), X]} \leq \vertiii{[j(A), X]} + \vertiii{[g(A), X]} \leq\sup_{x \in I} j(x) - \inf_{x \in I}j(x) + cg'(0) = f(c), 
\]
as desired.

We rephrase the condition on $A$ as a condition on $c$. Since the defining feature of $x_1$ is that $x_1>0$ and $f(x_1) = g(x_1)$, we see that $k(x_1) = m = F(c)$. This provides 
\[\vertiii{A} \leq x_1 = k^{-1}(F(c)).\] 
Since both $k$ and $F$ are strictly decreasing, this provides
\[F^{-1}(k(\vertiii{A})) \leq c.\]
The condition $c \leq \vertiii{A}$ is automatic from $A \geq 0$ and $\vertiii{X} \leq 1$.

The calculations for $f(x) = f_t(x)$ are immediate consequences of $k(x) = 1/(x+t)$, 
$F(x) = t/(x+t)^2$, and
$F^{-1}(x) = t^{1/2}x^{-1/2}-t$.

The calculations for $f(x) = x^r$ are immediate consequences of $k(x) = x^{r-1}$, 
$F(x) = rx^{r-1}$, and
$F^{-1}(x) = x^{1/(r-1)}r^{-1/(r-1)}$.

The only thing that remains to be shown is that $F^{-1}(k(x)) < x$ in general. Because $k(x)$ is the slope of the secant line from $(0, 0)$ to $(x, f(x))$, by the strict concavity of $f$, $F(x) = f'(x) < k(x)$. This implies $F^{-1}(k(x)) < x$.
\end{proof}
\begin{remark}
We stated this result with $f(0) = 0$ for simplicity. The result holds more generally with a redefinition of $k$. Moreover, because every concave function can be uniformly approximated by a concave $C^1((0, \infty))$ function made strictly concave on a compact interval, we see that this result also applies. As long as $f'$, existing except on a countable set, is strictly decreasing we still obtain an interval of values of $c=\vertii{[A,X]}$ in which the inequality holds that has positive length. 
\end{remark}
\begin{remark}
This proof indirectly makes use of an equation like
\[f(t) = \inf_{\lambda > 0}\left(f'(\lambda)t+(f(\lambda)-\lambda f'(\lambda))\right)\]
used in Section 2 of \cite{Ando FC OMF} for an operator monotone function $f$ but holds more generally. 

We also note that the acknowledgements of that paper mentions the fact that a certain result in that paper were originally proved for operator monotone functions but extended to hold under certain much weaker concavity assumptions. This is similar to the scenario that our Theorem \ref{LV extension} finds itself in.
\end{remark}

Note that the estimate we just proved can give estimates for some convex functions by manipulating them sufficiently. For instance, consider the function $f(x) = |x|$ and $A$ with $-a_1\leq A \leq a_2$ for $a_1, a_2 > 0$. If we define $\tilde f(x) = x+a_1-|x-a_1|$ and $B = A+a_1$, we can apply the above commutator inequality to obtain a commutator inequality for $f(x)$. Because this inequality can actually be proven in a more straight-forward way, we just state it as a proposition. Note that in the case that $\min(a_1,a_2)$ is much  smaller than $\max(a_1,a_2)$, this is quite an improvement of the trivial bound $\vertiii{[|A|, X]} \leq \|A\|\vertiii{X}$ and of (\ref{abs ineq}).
\begin{prop}\label{abs prop}
Suppose that $A, X$ are operators in $B(\H)$ with $\H$ separable, $-a_1 \leq A\leq a_2$ for $a_1, a_2 \geq 0$, and $\vertiii{X} \leq 1$. 
Then 
\[ \vertiii{[|A|, X]} \leq 2\min(a_1, a_2) + \vertiii{[A,X]}.\]
Hence, if $\vertiii{[A,X]} \geq  \min(a_1, a_2)$ then 
\[ \vertiii{[|A|, X]} \leq 3\vertiii{[A,X]}.\]
\end{prop}
\begin{proof}
By replacing $A$ with $-A$ if necessary, we may assume that $a_1 \leq a_2$. Note that $\min_{x \in [-a_1, a_2]}(x-|x|) = -2a_1$ and $\max_{x \in [-a_1, a_2]}(x-|x|) = 0$. Then by the argument that goes into (\ref{fund comm ineq}),  
\[\vertiii{[|A|, X]} \leq  \vertiii{[A-|A|, X]} + \vertiii{[A, X]} \leq 2\min(a_1, a_2) + \vertiii{[A,X]}.\]
\end{proof}
Note the tight inequality from \cite{OMC} for all $S, T$ self-adjoint:
\begin{align*}
\|\,|S|-|T|\,\| &\leq C_{abs}\|S-T\|\log\left(2+\log\left(\frac{\|S\|+\|T\|}{\|S-T\|}\right)\right) \\
&\leq C_{abs}\|S-T\|\log\left(2+\log\left(\frac{2\max(\|S\|,\|T\|)}{\|S-T\|}\right)\right).
\end{align*}
So, if $A$ and $X$ are as above then using $\Omega^\flat_{f,\mathfrak F} \leq 2\Omega_{f, \mathfrak F}$ for $\mathfrak F = [-a_1, a_2] \supset [-\|A\|, \|A\|]$, we have
\[\|\,[|A|,X]\,\| \leq 2C_{abs}\log\left(2+\log\left(\frac{2\vertii{A}}{\|[A,X]\|}\right)\right)\|[A,X]\|.\]

Compare also to (\ref{abs ineq}) which for the operator norm can be rewritten as:
\[\vertii{[|A|,X]} \leq \left(1 + \frac{\vertii{A}}{2\vertii{[A,X]}}\right)\vertii{[A,X]}.\]
So, we see that our rather trivial inequality of Proposition \ref{abs prop} can be an improvement in the case that $\min(a_1, a_2)$ is much smaller than $\max(a_1,a_2) \geq \vertii{A}$. 

The main downside of the inequality gotten in Theorem \ref{LV extension} is that it requires the commutator to not be small. However, this is a general feature of many commutator estimates of Lipschitz functions that are not commutator Lipschitz, including the two inequalities that we just stated above. See for instance Corollary 11.7 of \cite{OMC} which states that if $A$ is self-adjoint with $\sigma(A) \subset [a,b]$ and $\|X\| \leq 1$ then for all self-adjoint $B$,
\[\|[f(A)X-Xf(B)]\| \leq C_{abs}\|f\|_{\Lip(\R)}\log\left(2+\frac{b-a}{\|AX-XB\|}\right)\|AX-XB\|.\]
This was an improvement of the prior result in \cite{Farforovskaya Commutators of Functions in Perturbation Theory} which has the similar property that if one requires that the norm of a commutator or generalized commutator be bounded below then one obtains what is effectively a commutator Lipschitz bound for a Lipschitz function.
Our result, like that of Loring and Vides is an example of a similar type of inequality which provides simple concrete bounds. However, it has nothing to say in the case that the commutator is below that certain threshold, unlike the  above results that we just cited from the literature.

\vspace{0.05in}

We finish this section with noting that the approach of the previous theorem actually works without restriction when the commutator $\vertiii{-}$-Lipschitz constant equals the Lipschitz constant of Lipschitz functions. In that case, there is no need to restrict the spectrum of $A$ and we obtain Conjecture \ref{Conjecture 1} for all such unitarily-invariant norms.
\begin{thm}
Let $\H$ be a separable Hilbert space and $\vertiii{-}$ a unitarily invariant norm on $B(\H)$ such that $\vertiii{h}_{\CL([0, \infty))} = \vertii{h}_{\Lip([0, \infty))}$ for any Lipschitz function $h$ on $[0, \infty)$.

Suppose that $A, X$ are operators in $B(\H)$ with $A \geq 0$ and $\vertiii{X} \leq 1$. 
Let $f\geq 0$ be a concave function on $[0, \infty)$. Then
\[ \vertiii{[f(A), X]} \leq f\left(\vertiii{[A,X]}\right).\]
\end{thm}
\begin{proof} Without loss of generality, we may assume also that $f(0) = 0$. 
Without loss of generality we may assume also that $f$ is smooth and $f$ is strictly concave.

With $F(x) = f'(x)$, $F$ is strictly decreasing.
Let $c = \vertiii{[A,X]}$ and $m = F(c)$. Because $f$ is strictly concave, there is a unique point $x_1$ such that $f(x_1) = mx_1$. Define the piecewise function 
\[g(x) = \left\{\begin{array}{ll}
mx, & 0 \leq x < x_1 \\
f(x), & x \geq x_1\end{array}\right..\] 
Because $f$ is concave, $m \geq F(x) \geq 0$ for all $x \geq x_1$. In particular, $g$ is concave with $\|g\|_{\Lip([0,\infty))} = \|g'\|_{L^\infty([0,\infty))} = m$.

Let $j = f-g$. Then $j \geq 0$ on the interval $I = [0, x_1]$ and $j(x) = 0$ on $[x_1, \infty)$ . Because $j'$ is strictly decreasing on $I$, the maximum of $j$ occurs when $F(x) = m$, which is at $x_\ast = F^{-1}(m)=c$. 

So, by (\ref{fund comm ineq}), we see that
\begin{align*}
\vertiii{[f(A), X]}&\leq \left(\sup_{x \geq 0}(f-g)(x)-\inf_{x \geq 0}(f-g)(x)\right)\vertiii{X} +  \vertiii{g}_{\CL([0, \infty))}\|[A,X]\| \\
\leq
&\sup_{x \in I} j(x) - \inf_{x \in I}j(x) + mc= j(x_\ast)+mc=
f(c) - mc+mc = f(c).
\end{align*} 
\end{proof}
\begin{remark}
The Hilbert-Schmidt norm on $B(\H)$ for $\H$ separable has this property. As a consequence, when $\vertiii{-}$ is the Hilbert-Schmidt norm, this provides a short proof of the result (Theorem 3.3, \cite{Jocic}). Moreover, we even completely remove the condition that $f$ be operator monotone. 
\end{remark}

\newpage

\textbf{ACKNOWLEDGEMENTS}. The author would like to thank Eric A. Carlen for helpful feedback on this paper and Terry Loring for a helpful conversation on his work on this topic.

This research was partially supported by NSF grant DMS-2055282.

\end{document}